\documentclass[]{amsart}
\usepackage[utf8]{inputenc}
\usepackage{amsfonts}
\usepackage{amsmath,amscd}
\usepackage{amsthm}
\usepackage{latexsym}
\usepackage{amssymb}
\usepackage{enumerate}
\usepackage{url}
\markleft{\hfill G. Raposo\hfill }
\newtheorem{theorem}{Theorem}
\newtheorem*{theorem*}{Theorem}
\newtheorem{proposition}[theorem]{Proposition}
\newtheorem{corollary}[theorem]{Corollary}
\newtheorem{definition}[theorem]{Definition}
\newtheorem{lemma}[theorem]{Lemma}
\newtheorem{remark}[theorem]{Remark}
\newtheorem*{claim*}{Claim}
\usepackage{algorithm}
\usepackage{algpseudocode}

\usepackage{hyperref}

\usepackage{mathdots}

\usepackage{geometry}
\usepackage{tikz}
\usetikzlibrary{arrows,matrix,positioning}

\usepackage{caption} 
\usepackage{subcaption}
\usepackage{graphicx}

\newcommand{\Q}{\mathbb{Q}}
\newcommand{\R}{\mathbb{R}}

\newcommand{\ie}{{\it{i.$\,$e.\ }}}

\usetikzlibrary{decorations.pathreplacing,calligraphy} 

\DeclareMathOperator{\sign}{sign} 
\DeclareMathOperator{\s}{s} 

\newcommand{\sing}{\left[\substack{ \\ \oplus}\right]}
\newcommand{\hdoub}{\left[\substack{ \\- \oplus}\right]}
\newcommand{\vdoub}{\left[\substack{ - \\ \oplus}\right]}
\newcommand{\iquad}{\left[\substack{+ - \\- \oplus}\right]}
\newcommand{\tquad}{\ensuremath{\left[\substack{\oplus - \\- +}\right]}}

\newcommand{\ndquad}{\left[\substack{+ - \\\ominus +}\right]}

\newcommand{\nuquad}{\ensuremath{\left[\substack{+ \ominus \\- +}\right]}}

\newcommand{\nsing}{\left[\substack{ \\ \ominus}\right]}

\newcommand{\nhdoub}{\left[\substack{ \\ \ominus +}\right]}
\newcommand{\nvdoub}{\left[\substack{ \ominus \\ +}\right]}

\newcommand{\trip}{\ensuremath{\left[\substack{\oplus - \\- \hspace{2mm}}\right]}}

\newcommand{\dtrip}{\ensuremath{\left[\substack{+ - \\ \ominus \hspace{2mm}}\right]}}

\newcommand{\utrip}{\ensuremath{\left[\substack{+ \ominus \\- \hspace{2mm}}\right]}}

\usepackage{tikz}
\usepackage{circledsteps}
\usepackage{pstricks}

\title{A Refinement of Pohst's Inequality}
\author{Gabriel Raposo}
\address{Department of Statistics, University of California-Berkeley, \newline 367 Evans Hall, Berkeley, CA 94720, USA}
\email{raposo@berkeley.edu}

\makeatletter
\newenvironment{breakablealgorithm}
  {
   \begin{center}
     \refstepcounter{algorithm}
     \hrule height.8pt depth0pt \kern2pt
     \renewcommand{\caption}[2][\relax]{
       {\raggedright\textbf{\ALG@name~\thealgorithm} ##2\par}%
       \ifx\relax##1\relax 
         \addcontentsline{loa}{algorithm}{\protect\numberline{\thealgorithm}##2}%
       \else 
         \addcontentsline{loa}{algorithm}{\protect\numberline{\thealgorithm}##1}%
       \fi
       \kern2pt\hrule\kern2pt
     }
  }{
     \kern2pt\hrule\relax
   \end{center}
  }
\makeatother

\begin{document}
\begin{abstract}

We generalize an inequality conjectured by Pohst in 1977 and recently proved by the author and independently by Battistoni and Molteni. This new inequality improves a bound for the regulator in terms of the discriminant for totally real number fields by taking into account the signs of conjugates of a minimal unit. We give a new interpretation to the problem and exploit the combinatorial method used by Pohst.
\end{abstract}

	\maketitle
\section{Introduction}
In 1952 Remak \cite{Re} proved that the product $\prod_{1\leq i \leq j \leq n} \big|1-\prod_{k=i}^j x_k\big| $ is bounded above by $(n+1)^{(n+1)/2}$ when the variables $x_i$ are complex numbers with modulus $|x_i|\leq 1$. In 1977 Pohst \cite{Po} proved for $n \leq 10$ that the same product is bounded above by $2^{\lfloor \frac{n+1}{2}\rfloor}$ when the variables are real with absolute value $|x_i|\leq 1$. He gave a computer-assisted proof by using some elementary inequalities (see lemma \ref{Pohst}) and factorizing the product for each possible combination of signs for the variables $x_i$. 

In 1996 Bertin \cite{Be} attempted to give a proof for all $n$, unfortunately, the proof turned out to be incomplete. Recently the author \cite{Ra} and independently Battistoni and Molteni \cite{BM} found a proof for all $n$. In this article, we obtain a refinement of this inequality and obtain Pohst's inequality as a corollary. In particular, we prove that the factorization attempted computationally by Pohst in 1977 is possible.

\begin{theorem}\label{Theorem1}
Let $\{y_i\}_{i=1}^{n+1} \subset \R\backslash\{0\}$ such that $|y_i|\leq |y_{i+1}|$ for $i=1,\dots,n$. Let $$P_{n+1}(y_1,\dots,y_{n+1})=\prod_{1\leq i<j\leq n+1}\Big(1-\dfrac{y_i}{y_{j}}\Big),$$ 
let $p$ be the number of positive variables $y_i$ and $m=n+1-p$ be the number of negative variables, then $$P_{n+1}(y_1,\dots,y_{n+1})\leq 2^{\min(p,m)}.$$ 
\end{theorem}

A reformulation of the theorem via a change of variable (see section \ref{Section3}) gives

\begin{theorem}\label{Theorem2}
Let $\vec{x}=(x_1,\dots,x_{n})\in [-1,1]^{n}$, $x_i\neq 0$. Let 
$$f_n(\vec{x})=\prod_{1\leq i \leq j \leq n} (1-\prod_{k=i}^j x_k),$$ 
let $\alpha$ be the number positive values on the sequence $\{\prod_{k=1}^i x_i\}_{i=1}^{n}$ and let $\beta=n-\alpha$, then $$f_n(\vec{x}) \leq 2^{\min(\alpha+1,\beta)}.$$
\end{theorem}

Noticing that $\min(\alpha+1,\beta)\leq \lfloor \frac{n+1}{2}\rfloor$, we obtain

\begin{corollary}[Pohst's inequality]
Following the notations of theorem \ref{Theorem2}, then $$f_n(\vec{x}) \leq 2^{\lfloor \frac{n+1}{2}\rfloor}.$$
\end{corollary}

The number theoretic motivation is to improve the lower bound for the regulator $R_K$ in terms of the discriminant $D_K$ for number fields $K$. 

\begin{definition}
Let $K$ be a number field, we will say that $\varepsilon$ is a \textbf{minimal unit} if it is a unit with minimal non-zero length in the logarithmic lattice. That is,
$$m_K(\varepsilon) = \sum_{\omega} (\log||\varepsilon||_\omega)^2$$
is minimized among units with non-zero length. Here $\omega$ runs over the set of archimedean places of $K$ and $||\cdot||_\omega$ denotes the corresponding absolute value, normalized so that $|\text{Norm}_{K/\Q}(a)| = \prod_{\omega \in \infty_K} ||a||_\omega$ .
\end{definition}

\begin{corollary}\label{CorolNumbTheory}
Let $K$ a totally real primitive number field, $R_K$ its regulator and $D_K$ its discriminant, then
$$
\log|D_K|\leq \min(p,m) \log(4)+\sqrt{\gamma_{n-1}(n^3-n)/3}\,(\sqrt{n}R_K)^{1/(n-1)},
$$
where $n:=[K:\Q])$, $\gamma_{n-1}$ is Hermite's constant in dimension $n-1$, p is the number of positive conjugates of a minimal unit $\varepsilon$ and $m=n-p$.
\end{corollary}

The idea of the proof is to show that Pohst's algorithm \cite[Page 467]{Po} can find a factorization of the product such that we can apply lemma \ref{Pohst} below. In previous works \cite{Ra} \cite{BM} the existence of such a partition is proven for a subset of the factors that we call {\it{non-canonical}} (called {\it{wrong}} in \cite{BM}). In this article, we construct a partition for all factors. This also extends the methods employed in \cite{FR} where all possible sign configurations are taken into account to obtain better bounds when exactly two of the variables $y_i$ are complex conjugates. We hope these methods could shed some light on Battistoni's conjectures \cite{Ba}. See \cite{BM2} for recent progress in that direction. 

A heuristic approach to this problem consists in interpreting the variables $y_i$ as particles with either positive or negative charge ordered in the real line (see figure \ref{Heuristic}) while the factors $(1-y_i/y_j)$ are potentials that maximize when $|y_i|$ is close to $|y_j|$ with opposite charge and minimize when they have the same charge. This interpretation suggests that to maximize the product $\prod_{1\leq i < j \leq n} \big(1-y_i/y_j\big)$ we would need to pair each positive particle with a negative particle and should expect this to happen exactly $\min(p,m)$ times, our theorem proves that this is the case. 
\begin{figure}[h!] 
    \centering
\includegraphics[scale=0.25]{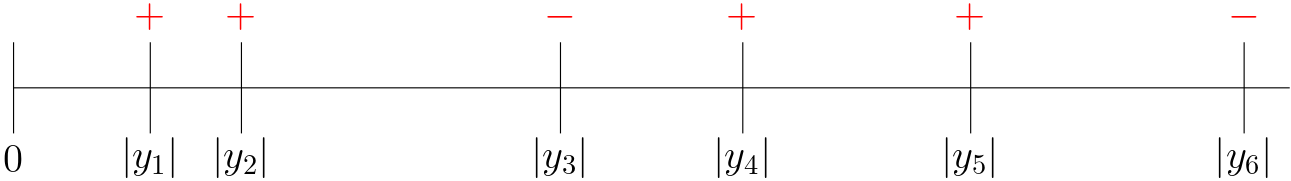}\\
\caption{Illustration of heuristic approach. }
        \label{Heuristic} 
\end{figure}

At the end of the article we propose a further improvement to the inequality motivated by this interpretation. However, note that the method employed in this paper exploits as far as possible the combinatorial tools introduced by Pohst. Any further improvement to find the maximum of these products would need more tools besides the basic four Pohst's inequalities of lemma \ref{Pohst}.

\section{Proof of corollary \ref{CorolNumbTheory} with examples}

We refer to \cite{ADF}[Section 2] for proofs of the following two propositions.

\begin{proposition}\label{Prop1}
    Let $K = \Q(\varepsilon)$ be a field of degree $n$ and signature $(r_1, r_2)$ with $\varepsilon \in \mathcal{O}_K^*$ and discriminant $D_K$. Let $m_K(\varepsilon)$ be the length of $\varepsilon$ in the logarithmic lattice of the units of $K$. Let $\varepsilon_1, . . . , \varepsilon_n$ be the conjugates of $\varepsilon$ ordered in increasing absolute value. Then
$$\log |D_K | \leq 2 \log(P_{n} (\varepsilon_1,\dots, \varepsilon_n)) + m_K(\varepsilon)\cdot \sqrt{\frac{n^3-n-4r_2^3-2r_2}{3}}.$$
    
\end{proposition}

\begin{proposition} \label{Prop2}
Assume $K = \Q(\varepsilon)$ as above and that $\varepsilon$ is a minimal unit with length $m_K(\varepsilon)$. Let $r=r_1+r_2-1$, then
$$ m_K(\varepsilon)\leq \big( \sqrt{r+1} R_K \gamma_r^{r/2}\big)^{1/r},$$
where $\gamma_j$ is the Hermite constant of dimension $j$ and $R_K$ is the regulator of $K$. 
\end{proposition}

\begin{proof}[Proof of corollary \ref{CorolNumbTheory}]
Starting from proposition \ref{Prop1}, in this setting $K$ is a totally real primitive field, hence $r_1=n$, $r_2=0$ and $r=n-1$.
We apply both theorem $1$ to bound the term $P_{n} (\varepsilon_1,\dots, \varepsilon_n)$ and proposition \ref{Prop2} to bound $m_K(\varepsilon)$.
\end{proof}

It is natural to wonder what the behavior of $\min(p, m)$ is for a minimal unit $\varepsilon$. While it is computationally challenging to find a fundamental system of units and to find a shortest length vector in the logarithmic lattice, a low degree setting provides a toy model where we can study the frequency of $\min(p, m)$.\begin{table}[h]
    \begin{minipage}{.5\linewidth}
      \centering
        \begin{tabular}{|c|c|c|c|}
        \hline
$D_k$  & $P(x)$ & $u$ & $\min(p,m)$     \\
\hline
$ 5 $ & $x^2 - x - 1 $ & $ a $ & $ 1 $ \\
$ 8 $ & $x^2 - 2 $ & $ a + 1 $ & $ 1 $ \\
$ 12 $ & $x^2 - 3 $ & $ a - 2 $ & $ 0 $ \\
$ 13 $ & $x^2 - x - 3 $ & $ a + 1 $ & $ 1 $ \\
$ 17 $ & $x^2 - x - 4 $ & $ 2a + 3 $ & $ 1 $ \\
\hline
\end{tabular}
    \end{minipage}%
    \begin{minipage}{.5\linewidth}
      \centering
        \begin{tabular}{|c|c|c|c|}
        \hline
$D_k$  & $P(x)$ & $u$ & $\min(p,m)$     \\
\hline
$ 21 $ & $x^2 - x - 5 $ & $ a + 2 $ & $ 0 $ \\
$ 24 $ & $x^2 - 6 $ & $ 2a - 5 $ & $ 0 $ \\
$ 28 $ & $x^2 - 7 $ & $ 3a - 8 $ & $ 0 $ \\
$ 29 $ & $x^2 - x - 7 $ & $ a + 2 $ & $ 1 $ \\
$ 33 $ & $x^2 - x - 8 $ & $ 8a + 19 $ & $ 0 $ \\
\hline
\end{tabular}
    \end{minipage} 
    \caption{Minimal units for totally real quadratic fields.} \label{Table11}
\end{table}

The quadratic case is particularly well behaved since a fundamental unit is also a minimal unit. In this setting there are two possibilities, either both conjugates of the minimal unit have same sign (hence $\min(p, m)=0$) or they have opposite sign (hence $\min(p, m)=1$), this can be checked by computing the norm of a fundamental unit. Over all $303\,957$ totally real quadratic number fields with discriminant smaller than $10^6$, a total of $228\,045$ (Around 75.02\%) have the conjugates of a minimal unit with same sign. For easier visualization, we provide a table (see table \ref{Table11}) with the first $10$ totally real quadratic number fields with smallest discriminant, $D_K$ denotes the discriminant, $P(x)$ is the defining polynomial, $u$ is a minimal unit and $a$ is a field generator.

To perform the computation we used SageMath (see \cite{Sa}) to enumerate totally real quadratic fields by discriminant, compute a fundamental unit $u$ and compute the minimal polynomial $m_u(x)$ of $u$. Note that the constant term of such a minimal polynomial $m_u(0)$ is either $1$ or $-1$. If it is $1$ then both conjugates have the same sign and $\min(p,m)=0$, in the other case $\min(p,m)=1$. 

\section{Proof of the theorems}\label{Section3}

We start by proving that theorem \ref{Theorem1} and \ref{Theorem2} are equivalent.

\begin{proof}[Proof of equivalence between theorems \ref{Theorem1} and \ref{Theorem2}] We prove each direction separately. 
\begin{itemize}
    \item \textbf{Theorem \ref{Theorem1} implies theorem \ref{Theorem2}.} Take $y_1$ to be positive and let  $x_i=y_i/y_{i+1}$ for $i=1,\dots,n$. Then $\alpha$ and $\beta$ be the number positive and respectively negative terms in the sequence $\{\prod_{k=1}^i x_i\}_{i=1}^{n}$. Via this change of variables we get $P_{n+1}(y_1,\dots,y_{n+1})=f_n(\vec{x})$, $p=\alpha+1$ and $m=\beta$.
    \item \textbf{Theorem \ref{Theorem2} implies theorem \ref{Theorem1}.} Take $x_i=y_i/y_{i+1}$ for $i=1,\dots,n$. Here $\alpha$ is counting the number of variables $y_2,\dots, y_{n+1}$ with same sign as $y_1$ while $\beta$ is counting the number of variables with opposite sign to $y_1$. If $y_1$ is positive we get $p=\alpha+1$ and $m=\beta$, if $y_1$ is negative we get $p=\beta$ and $m=\alpha+1$. In both cases we obtain the desired inequality.\qedhere
\end{itemize}
\end{proof}

We devote the rest of the section to prove theorem \ref{Theorem2} using Pohst's original inequalities.

\begin{lemma}\label{Pohst}
(Pohst)
\begin{enumerate}
    \item If $a \in [-1,1]$, then $(1-a) \leq 2$.
    \item If $a \in [0,1]$ and $b \in [-1,0]$, then $(1-a)(1-ab) \leq 1$.
    \item If $a,b \in [-1,1]$, then $(1-a)(1-b)(1-ab) \leq 2$.
    \item If $a \in [0,1]$ and $b,c \in [-1,0]$, then $(1-a)(1-ab)(1-ac)(1-abc)\leq 1$.
\end{enumerate}
\end{lemma}
\begin{proof}
See \cite[p. 468]{Po}.
\end{proof}

\begin{definition}
Let $\vec{x}=(x_1,\dots,x_n)\in \big([-1,1]\backslash\{0\}\big)^n$ and $a_{\vec{x}(i,j)}:=1-\prod_{k=i}^j x_{k}$. Define the \textbf{product sign} $s$ of $a_{\vec{x}(i,j)}$ by $$s(a_{\vec{x}(i,j)}):= \sign(1-a_{\vec{x}(i,j)}) = \prod_{k=i}^j \sign(x_{k}).$$ 
\begin{enumerate}
    \item We say that $a_{\vec{x}(i,j)}=1-\prod_{k=i}^j x_k$ is \textbf{canonical} if $s(a_{\vec{x}(i,j)})=(-1)^{i+j+1}$.
    \item We say that $a_{\vec{x}(i,j)}=1-\prod_{k=i}^j x_k$ is \textbf{non canonical} if $s(a_{\vec{x}(i,j)})=(-1)^{i+j}$.
    \item Define the set $K_{\vec{x}}$ of \textbf{canonical indices} of $\vec{x}$, $$K_{\vec{x}}:=\{(i,j) : s(a_{\vec{x}(i,j)}) = (-1)^{i+j+1}\}.$$
    \item Define the set $J_{\vec{x}}$ of\textbf{ non canonical indices} of $\vec{x}$, $$J_{\vec{x}}:=\{(i,j) : s(a_{\vec{x}(i,j)}) = (-1)^{i+j}\}.$$
\end{enumerate}
\end{definition}

The terms $a_{\vec{x}(i,j)}$ and the indices $(i,j)$ can be ordered in a triangular way (see figure \ref{figure4}). We will partition  $K_{\vec{x}}$ and $J_{\vec{x}}$ into subsets of 1, 2, 3, or 4 elements so that we can apply the corresponding products cases (1), (2), (3), or (4) respectively of lemma \ref{Pohst}. 

\begin{definition} \label{defpartition}
Let $\vec{x}=(x_1,\dots,x_n)\in \big([-1,1]\backslash\{0\}\big)^n$ and let $\pi_{\vec{x}}$ be a set partition of $J_{\vec{x}}$. We say that $\pi_{\vec{x}}$ is a \textbf{good partition} if for each $p \in \pi_{\vec{x}}$ then one of the following holds.
\begin{enumerate}
    \item $p=\{(i,j)\}$ with $s(a_{\vec{x}(i,j)})=1$.
    \item $p=\{(i,j),(i',j')\},\, s(a_{\vec{x}(i,j)})=1, \, s(a_{\vec{x}(i',j')})=-1$ and either $i'\leq i$ and $j=j'$, or $i'=i$ and $j\leq j'$. See figures \ref{picturedefpart} (A) and \ref{picturedefpart} (B).
    \item $p=\{(i,j),(i-l,j),(i,j+l'),(i-l,j+l')\}, \, s(a_{\vec{x}(i,j)})= s(a_{\vec{x}(i-l,j+l')}) =1 , \, s(a_{\vec{x}(i-l,j)})=s(a_{\vec{x}(i,j+l')})=-1$, and $l,l'\geq 1$. See figure \ref{picturedefpart} (C).
\end{enumerate}
\end{definition}
\vspace{-1cm}
\begin{figure}[h]
    \centering
    \begin{subfigure}[t]{0.45\textwidth}
    \begin{tikzpicture}
    \matrix [matrix of math nodes] (m)
        {
           & \hspace{1.5cm} & (i',j)_{-} &\hdots &(i,j)_{+}   \\
           & &\hspace{1mm} & \\
           & &\hspace{1mm} & \\
        };  
\end{tikzpicture} \caption{Configuration \textit{(2)} horizontally.}   
    \label{Conf2h}
    \end{subfigure}
    \hfill \hfill
    \begin{subfigure}[t]{0.45\textwidth}
        \centering
    \begin{tikzpicture}
    \matrix [matrix of math nodes] (m)
        {
             &  &(i,j')_{-} &  &    \\
            &  & \vdots & &  \\
            &  & (i,j)_{+} & & \\
        };  
\end{tikzpicture}    \caption{Configuration \textit{(2)} vertically.}
    \label{Conf2v}
    \end{subfigure}
    \vspace{0.2cm}

    \centering
    \begin{subfigure}[t]{0.45\textwidth}
        \centering
     \begin{tikzpicture}
    \matrix [matrix of math nodes] (m)
        {
           (i-l,j+l')_+ & \hdots & (i,j+l')_-\\
           \vdots & & \vdots\\
           (i-l,j)_- & \hdots & (i,j)_+\\
        };  
\end{tikzpicture}    \caption{Configuration \textit{(3)}.}
    \label{Conf3}
    \end{subfigure}
    \hfill
   \begin{subfigure}[t]{0.45\textwidth}
    \centering
    \begin{tikzpicture}
    \matrix [matrix of math nodes] (m)
        {
           (i,j)_+ & \hdots & (i+l,j)_-\\
           \vdots & & \\
           (i,i-l-1)_- &  & \\
        };  
\end{tikzpicture} \caption{Configuration \textit{(5)}.}   
    \label{Conf5}
    \end{subfigure}
    
    \caption{Configurations from definitions \ref{defpartition} and \ref{defpartition2}.}
        \label{picturedefpart}
\end{figure}

\begin{definition} \label{defpartition2}
Let $\vec{x}=(x_1,\dots,x_n)\in \big([-1,1]\backslash\{0\}\big)^n$ and let $\eta_{\vec{x}}$ be a set partition of $K_{\vec{x}}$. We say that $\eta_{\vec{x}}$ is a \textbf{good partition} if for each $p \in \eta_{\vec{x}}$ then one of the following holds.
\begin{enumerate}
    \item\hspace{-2mm}, (2), or (3) as in definition \ref{defpartition}.
   \setcounter{enumi}{3}
    \item $p=\{(i,j)\}$ with $s(a_{\vec{x}(i,j)})=-1$.
    \item $p=\{(i,j),(i+l,j),(i,i+l-1)\}, \, s(a_{\vec{x}(i,j)})=1 , \, s(a_{\vec{x}(i+l,j)})=s(a_{\vec{x}(i,i+l-1)})=-1$, and $l\geq 1$. See figure \ref{picturedefpart} (D).
\end{enumerate}
We additionally require that the number of elements $p$ in $\eta_{\vec{x}}$ falling in points (4) or (5) of the definition is equal to $\min(\alpha+1,\beta)$.
\end{definition}

Notice that $p$ is a singleton in item {\it{(1)}} of definition \ref{defpartition} or item {\it{(4)}} of definition \ref{defpartition2}. Each $p$ as in definitions \ref{defpartition} and \ref{defpartition2} can be represented graphically, see figure \ref{picturedefpart}.

\begin{lemma}\label{Inequalities}
Let $\vec{x}=(x_1,\dots,x_n)\in \big([-1,1]\backslash\{0\}\big)^n$, let let $\pi_{\vec{x}}$ be a good partition of $J_{\vec{x}}$ and let $\eta_{\vec{x}}$ be a good partition of $K_{\vec{x}}$. Let $p\in \pi_{\vec{x}}$ or $p \in \eta_{\vec{x}}$.
\begin{enumerate}
    \item If $p$ is as in (1), (2) or (3) of definitions \ref{defpartition} or \ref{defpartition2}, then 
    $\prod_{(i,j) \in p} a_{\vec{x}(i,j)} \leq 1.$
    \item if $p$ is as in (4) or (5) of definition \ref{defpartition2}, then 
    $\prod_{(i,j) \in p} a_{\vec{x}(i,j)} \leq 2.$
\end{enumerate}
\end{lemma}
\begin{proof}
We use the inequalities provided by lemma \ref{Pohst}. We will go case by case.
\begin{enumerate}
    \item $\prod_{(i,j) \in p} a_{\vec{x}(i,j)} = 1-\prod_{k=i}^j x_{k}$. Since $\prod_{k=i}^j x_{k}>0$, we obtain the desired bound. 
    \item $\prod_{(i,j) \in p} a_{\vec{x}(i,j)} = (1-\prod_{k=i}^j x_k)(1-\prod_{k=i-l}^j x_k)$ or $\prod_{(i,j) \in p} a_{\vec{x}(i,j)} = (1-\prod_{k=i}^j x_k)(1-\prod_{k=i}^{j+l'} x_k)$ for some index $(i,j)$ and $l,l'>0$. In each case, inequality (2) of lemma \ref{Pohst} with $a= \prod_{k=i}^j x_k$ and $b=\prod_{k=i-l}^{i-1} x_k$ or $b=\prod_{k=j+1}^{j+l'} x_k$ respectively provides the desired bound. 
    \item $\prod_{(i,j) \in p} a_{\vec{x}(i,j)} = (1-\prod_{k=i}^j x_k)(1-\prod_{k=i-l}^j x_k)(1-\prod_{k=i}^{j+l'} x_k)(1-\prod_{k=i-l}^{j+l'} x_k)$ for some index $(i,j)$ and $l,l'>0$. Inequality (4) of lemma \ref{Pohst} with $a= \prod_{k=i}^j x_k$, $b=\prod_{k=i-l}^{i-1} x_k$ and $c=\prod_{k=j+1}^{j+l'} x_k$ provides the desired bound. 
    \item $\prod_{(i,j) \in p} a_{\vec{x}(i,j)} = 1-\prod_{k=i}^j x_{k}$. Inequality (1) of lemma \ref{Pohst} with $a = \prod_{k=i}^j x_{k}$  provides the desired bound. 
    \item $\prod_{(i,j) \in p} a_{\vec{x}(i,j)} = (1-\prod_{k=i}^j x_{k})(1-\prod_{k=i+l}^j x_{k})(1-\prod_{k=i}^{i+l-1} x_{k})$ for some index $(i,j)$ and $l>0$. Inequality (3) of lemma \ref{Pohst} with $a=\prod_{k=i}^{i+l-1} x_{k}$ and $b=\prod_{k=i+l}^{j} x_{k}$ provides the desired bound. \qedhere
\end{enumerate}
\end{proof}

\begin{lemma}\label{partition}
For every $\vec{x}=(x_1,\dots,x_n) \in \big([-1,1]\backslash\{0\}\big)^n$ the set $J_{\vec{x}}$ has a good partition.
\end{lemma}
\begin{proof}
A full proof can be found in \cite[Lemma 8]{Ra}.
\end{proof}

\begin{lemma}\label{partition2}
For every $\vec{x}=(x_1,\dots,x_n) \in \big([-1,1]\backslash\{0\}\big)^n$ the set $K_{\vec{x}}$ has a good partition.
\end{lemma}

The proof of lemma \ref{partition2} is the main technical challenge and is postponed to the next section. We can now prove the main theorem. 

\begin{proof}[Proof of main theorem]
We regroup the terms in the product $f_n$ so that we can apply lemma \ref{Pohst}. Every index is either canonical or non canonical, we record this as \\ 

\noindent\textbf{Claim 1.} For each pair $1\leq i \leq j \leq n$ we have $(i,j) \in  K_{\vec{x}}$ or $(i,j) \in  J_{\vec{x}}$.\\

\noindent
lemma \ref{partition} and lemma \ref{partition2} ensure the existence of good partitions $\pi_{\vec{x}}$ and $\eta_{\vec{x}}$ on $J_{\vec{x}}$ and $K_{\vec{x}}$ respectively.
\begin{align*}
    f_n(x_1,\dots,x_n) & :=    \prod_{1\leq i \leq j \leq n} a_{\vec{x}(i,j)} && 
    (\text{By definition of }f_n)\\
     &=  \Big( \prod_{(i,j) \in K_{\vec{x}}} a_{\vec{x}(i,j)} \Big)\cdot  \Big( \prod_{(i,j)\in J_{\vec{x}}} a_{\vec{x}(i,j)} \Big) && \text{(By claim 1)}\\
    & =   \Big( \prod_{p \in \eta_{\vec{x}}} \prod_{(i,j) \in p} a_{\vec{x}(i,j)} \Big) \cdot \Big( \prod_{p \in \pi_{\vec{x}}} \prod_{(i,j) \in p} a_{\vec{x}(i,j)} \Big) && \text{(By lemmas \ref{partition} and \ref{partition2})} \\
    & \leq    2^{\min(\alpha+1,\beta)}. && \text{(By lemma \ref{Inequalities})}
\end{align*}
To prove the final inequality we need a bound for the number of $p$ in cases {\it{(4)}} or {\it{(5)}} of definition \ref{defpartition2}. The definition of a good partition ensures that it is at most $\min(\alpha+1,\beta)$.
\end{proof}

\section{Proof of lemma \ref{partition2}}

We provide an algorithm that constructs a good partition $\eta_{\vec{x}}$ of $K_{\vec{x}}$ inductively. We start by introducing some preliminary notation, propose a total order on $K_{\vec{x}}$ and provide some operations over sets to construct the partitions and introducing some preliminary definitions. We finally explicitly show the algorithm and prove its correctness. 

For simplicity, we are going to say that a pair $(i,j)$ is positive (respectively negative) if the corresponding term $a_{\vec{x}(i,j)}$ has positive (respectively negative) product sign. If we already know the product sign of a pair we will add it as a sub-index.

The following fact, whose proof is a straightforward calculation, ensures that through our proof we remain inside the set $K_{\vec{x}}$.\\

\textbf{Fact:} If in the set $\{(i, j),(i-l, j),(i, j+l'),(i-l, j+l')\}$, where $l, l' > 1$, three of the pairs are canonical, then the last one is also canonical. Moreover we have the relation $s(a_{\vec{x}(i,j)}) s(a_{\vec{x}(i-l,j+l')}) = s(a_{\vec{x}(i-l,j)}) s(a_{\vec{x}(i,j+l')})$. \\

In particular, given three of such signs we can compute the forth one. 

\subsection{Order on $K_{\vec{x}}$}

Let us define a total order over $\{(i,j): i\leq j\}$ as in Figure \ref{figure4}.

\begin{definition}
Given $(i,j), (i',j') \in \{(i,j): i\leq j\}$, we say that $(i,j)$ is \textbf{smaller} than $(i',j')$ if and only if $j' > j$, or $j'=j$ and $i' < i$. We denote this relation $(i,j)  \rightarrow (i',j')$.
\end{definition}

\begin{remark}
It may be simpler to notice that the order $\rightarrow$ corresponds to the lexicographic order on the pairs $(j,-i)$.
\end{remark}
\begin{figure}[h] 
    \centering
    \vspace{-5mm}
    \begin{tikzpicture}
    \matrix [matrix of math nodes] (m)
        {
            (1,n) &(2,n) &(3,n) &(4,n) &(5,n) &(6,n) & &(n,n)  \\
            \vdots& \vdots &\vdots &\vdots &\vdots &\vdots & \iddots & \\
            (1,6) &(2,6) &(3,6) &(4,6) &(5,6) &(6,6) & &  \\    
            (1,5) &(2,5) &(3,5) &(4,5) &(5,5) & & & \\      
            (1,4) &(2,4) &(3,4) &(4,4) & & & & \\     
            (1,3) &(2,3) &(3,3) & & & & & \\    
            (1,2) &(2,2) & & & & & & \\     
            (1,1)\\ 
            \hspace{1cm}&  \hspace{1cm}& \hspace{1cm} & \hspace{1cm} & \hspace{1cm} & \hspace{1cm}\\
        };  
       \draw[color=red][->] (m-7-2) edge (m-7-1);
       \draw[color=red][->]  (m-6-3) edge (m-6-2);
       \draw[color=red][->]  (m-6-2) edge (m-6-1);
       \draw[color=red][->]  (m-5-4) edge (m-5-3);
       \draw[color=red][->]  (m-5-3) edge (m-5-2);
       \draw[color=red][->]  (m-5-2) edge (m-5-1);
       \draw[color=red][->]  (m-4-5) edge (m-4-4);
       \draw[color=red][->]  (m-4-4) edge (m-4-3);
       \draw[color=red][->]  (m-4-3) edge (m-4-2);
       \draw[color=red][->]  (m-4-2) edge (m-4-1);
       \draw[color=red][->]  (m-3-6) edge (m-3-5);
       \draw[color=red][->]  (m-3-5) edge (m-3-4);
       \draw[color=red][->]  (m-3-4) edge (m-3-3);
       \draw[color=red][->]  (m-3-3) edge (m-3-2);
       \draw[color=red][->]  (m-3-2) edge (m-3-1);
       \draw[color=red][->]  (m-1-6) edge (m-1-5);
       \draw[color=red][->]  (m-1-5) edge (m-1-4);
       \draw[color=red][->]  (m-1-4) edge (m-1-3);
       \draw[color=red][->]  (m-1-3) edge (m-1-2);
       \draw[color=red][->]  (m-1-2) edge (m-1-1);
       \draw[color=red][->]  (m-1-8) edge[dotted] (m-1-6);
       \draw[color=red][->]  (m-8-1) edge[out=170,in=350] node[yshift=0.3ex] { } (m-7-2);
       \draw[color=red][->]  (m-7-1) edge[out=170,in=350] node[yshift=0.3ex] { } (m-6-3);
       \draw[color=red][->]  (m-6-1) edge[out=170,in=350] node[yshift=0.3ex] { } (m-5-4);
       \draw[color=red][->]  (m-5-1) edge[out=170,in=350] node[yshift=0.3ex] { } (m-4-5);
       \draw[color=red][->]  (m-4-1) edge[out=170,in=350] node[yshift=0.3ex] { } (m-3-6);
       \draw[color=red][->]  (m-3-1) edge[out=170,in=350,dotted] node[yshift=0.3ex] { } (m-2-7);
       \draw[color=red][->]  (m-2-1) edge[out=170,in=350,dotted] node[yshift=0.3ex] { } (m-1-8);
\end{tikzpicture}
    \vspace{-10mm}
    \caption{Order used to construct the partition of $K_{\vec{x}}$.}
    \label{figure4}
\end{figure}

\noindent
Define $P_0 \subset \mathcal{P}(K_{\vec{x}})$, a subset of the power set of $K_{\vec{x}}$, as $$P_0:=\big\{ \{(i,j)\} \big| \, s(a_{\vec{x}(i,j)})=1=(-1)^{i+j+1},\  1\leq i\leq j\leq n \big\}.$$

Let $N\geq 0$ be the number of negative pairs $(i,j)\in K_{\vec{x}}$. If $N=0$, then $P_0$ is already a good partition. If $N>0$, for $1\leq k \leq N$ we will add inductively to $P_{k-1}$ the $k$-th negative pair $(i,j)_-$.  Thus each of the first $k$ negative pairs $(i,j)_-$ in the order $ \rightarrow$ is contained in (some element of) $P_k$. We will show that $P_N$ is a partition $\eta_{\vec{x}}$ of $K_{\vec{x}}$ whose existence is claimed  by lemma \ref{partition2}. 

\subsection{Notation}

We will often need to refer to pairs contained in partition classes with specific configurations. We will introduce some auxiliary notation to simplify the language of the proof. For instance, given $(i,j) \in K_{\vec{x}}$, we draw a small diagram of the form $\left[\substack{+ - \\- +}\right]$ to describe the class partition it is contained in. We indicate by a circle the position of the pair $(i,j)$. This is described with greater precision on the following definition. 

\begin{definition}\label{cases}
Let $0 \leq k \leq N$. If $(i,j)_+ \in K_{\vec{x}}$ is positive, then we will say that
\begin{enumerate}
    \item $(i,j)_+$ is in a singleton if $\{(i,j)_+\} \in P_k$. We call it a $\sing$-configuration.

    \item $(i,j)_+$ is in an h-doubleton if $\{(i,j)_+,(i-l,j)_-\} \in P_k$, $i-l < i$. We call it a  $\hdoub$-configuration.

    \item $(i,j)_+$ is in a v-doubleton if $\{(i,j)_+,(i,j+l)_-\} \in P_k$, $j<j+l$. We call it a  $\vdoub$-configuration.

    \item $(i,j)_+$ is in an i-quadrupleton if $\{(i,j)_+,(i-l_1,j)_-,(i,j+l_2)_-,(i-l_1,j+l_2)_+\} \in P_k$, $l_1,l_2 \geq 1$. We call it a  $\iquad$-configuration.

    \item $(i,j)_+$ is in a t-quadrupleton if $\{(i,j)_+,(i+l_1,j)_-,(i,j-l_2)_-,(i+l_1,j-l_2)_+\} \in P_k$, $l_1,l_2 \geq 1$ and $i+l_1\leq j-l_2$.  We call it a $\tquad$-configuration.
    \item $(i,j)_+$ is in a tripleton if $\{(i,j)_+,(i,j-l)_-,(i+l',j)_-\}\in P_k$, $l,l'\geq 1$. We call it a $\trip$-configuration.

\end{enumerate}

If $(i,j)_-$ is negative and is contained in some subset of $P_k$, then we will say that

\begin{enumerate}
    \item $(i,j)_-$ is in an h-doubleton if $\{(i,j)_-,(i+l,j)_+\} \subset p \in P_k$, $i<i+l\leq j$.  We call it a $\nhdoub$-configuration.
    \item $(i,j)_-$ is in a v-doubleton if $\{(i,j)_-,(i,j-l)_+\} \subset p \in P_k$, $i\leq j-l< j$. We call it a $\nvdoub$-configuration.
    \item $(i,j)_-$ is in a d-tripleton if $\{(i,j)_-,(j+1,l)_-,(i,l)_+\}\in P_k$, $l\geq j+1$. We call it a $\dtrip$-configuration. 
    \item $(i,j)_-$ is in a u-tripleton if $\{(i,j)_-,(l,i-1)_-,(l,j)_+\}\in P_k$, $l\leq i-1$. We call it a $\utrip$-configuration. 
\end{enumerate}

\end{definition}

For example, saying that $(i,j)_-$ is in $\dtrip$-configuration means that $\{(i,j)_-,(j+1,l)_-,(i,l)_+\}\in P_k$ for some $l\geq 1$.

\subsection{Set Operations}

\begin{figure}[h]
    \centering
    \vspace{-5mm}
    \begin{subfigure}[t]{0.45\textwidth}
    \begin{tikzpicture}\hspace{1.5cm}
    \matrix [matrix of math nodes] (m)
        {
            (i,j)_{-} &\hdots &(i',j')_{+}  \\
            \vdots& & \\
            (i',j')_{+}& & \\
        };  
\end{tikzpicture} \caption{Two cases of Operation 1.}   
    \label{Op1fig}
    \end{subfigure}
    \hfill \hfill
    \begin{subfigure}[t]{0.45\textwidth}
        \centering
    \begin{tikzpicture}
    \matrix [matrix of math nodes] (m)
        {
            (r,j)_{+} &\hdots &(i,j)_{-}  \\
            \vdots&  &\vdots \\
            (r,l)_{-}& \hdots &(i,l)_{+} \\
        };    
\end{tikzpicture}    \caption{Operation 2.}
    \label{Op2fig}
    \end{subfigure}
    \vspace{0.2cm}

    \centering
    \begin{subfigure}[t]{0.45\textwidth}
        \centering
     \begin{tikzpicture}
    \matrix [matrix of math nodes] (m)
        {
         &(i,j)_-\\
         &\vspace{1cm} \\
        };
\end{tikzpicture}    \caption{Operation 3.}
    \label{Op3fig}
    \end{subfigure}
    \hfill
   \begin{subfigure}[t]{0.45\textwidth}
    \centering
    \begin{tikzpicture}
    \matrix [matrix of math nodes] (m)
        {
           (i',j)_+ & \hdots & (i,j)_-\\
           \vdots & & \\
           (i',i-1)_- &  & \\
        };  
\end{tikzpicture} \caption{Operation 4.}   
    \label{Op4fig}
    \end{subfigure}
    \caption{Operations over $P_{k-1}$ to produce $P_k$.}
        \label{figureoP}
\end{figure}

If $(i,j)_-$ is the $k$-th  negative element of $K_{\vec{x}}$ we will choose one of four operations to apply to $P_{k-1}$ to produce $P_k$. See figure \ref{figureoP}.\\

\begin{enumerate}
    \item[ ] Operation 1. If $\{(i',j')\} \in P_{k-1}$ with $i'=i$ and $i \leq j' < j$, or $j'=j$ and $i < i' \leq j$, where $(i',j')_+$ is positive, then
$$
P_k:=\Big(P_{k-1}-\big\{\{(i',j')\}\big\}\Big)\cup\Big\{\{(i',j'),(i,j)\}\Big\}.
$$
    \item[ ] Operation 2. If  $\{(i,l),(r,l)\}\in P_{k-1}$ and $\{(r,j)\} \in P_{k-1}$ with $i\leq l < j$ and $1 \leq r < i$, where $(i,l)_+$ and $(r,j)_+$ are positive and $(r,l)_-$ is negative, then 
$$
P_k:=\Big(P_{k-1}-\big\{\{(i,l),(r,l)\},\{(r,j)\}\big\}\Big)\cup\Big\{\{(i,l),(r,l),(i,j),(r,j)\}\Big\}.
$$

   \item[ ] Operation 3. 
   $$
P_k:=P_{k-1}\cup\Big\{\{(i,j)\}\Big\}.
$$
   \item[ ] Operation 4. If  $\{(i',i-1)\}\in P_{k-1}$ and $\{(i',j)\} \in P_{k-1}$ with $1\leq  i' < i$, where $(i',i-1)_-$ is negative and $(i',j)_+$ is positive, then 
   $$
P_k:=\Big(P_{k-1}-\big\{\{(i',i-1)\},\{(i',j)\}\big\}\Big)\cup\Big\{\{(i',i-1),(i',j),(i,j)\}\Big\}.
$$
\end{enumerate}

Thus operation 1 removes a singleton from $P_{k-1}$ and inserts a doubleton containing this singleton. Operation 2 removes a doubleton and a singleton from $P_{k-1}$ and inserts a quadrupleton containing the removed elements and forming the vertices of a rectangle. Operation 3 adds a singleton. Operation 4 removes two singletons and adds a tripleton.

At step $0$, $P_0$ only contains singletons with indices with positive sign, these correspond to partition classes of the form {\it{(1)}} in definition \ref{defpartition2}. Each time we use one of the operations we are only left with partitions classes as in definition \ref{defpartition2}. For instance, operation 1 produces classes of the form of {\it{(2)}} in definition \ref{defpartition2}, operation 2 produces classes of the form of {\it{(3)}} in definition \ref{defpartition2}, operation 3 produces classes of the form of {\it{(4)}} in definition \ref{defpartition2} and operation 4 produces classes of the form of {\it{(5)}} in definition \ref{defpartition2}.

\subsection{The algorithm}

For each $1\leq k \leq N$, suppose $(i,j)_-$ the $k$-th negative pair we will follow the next algorithm to decide which operation to use to produce $P_k$ from $P_{k-1}$. Fix the horizontal list  $$l_{(i,j)}:=\{(i,j),(i+1,j),\dots,(j,j)\}\cap K_{\vec{x}}.$$
\begin{definition}
Let $1 \leq j \leq n$ and let $\alpha'_j$ be the number positive pairs on the sequence $\{(1,i)\}_{i=1}^{j}$ and let $\beta'_j=j-\alpha'_j$ the number of negative pairs.
\begin{itemize}
    \item We say that the \textbf{level j is stable} if $\min(\alpha'_{j-1}+1,\beta'_{j-1})=\min(\alpha'_j+1,\beta'_j)$. Otherwise, we say that the \textbf{level j is unstable}. We say that $(i,j)$ is at a stable level if level j is stable and $(i,j)$ is at a unstable level if level j is unstable. 
    \item For each $1\leq k \leq N$, if $(i,j)_-$ is the $k$-th negative pair and there exist some positive pair on $l_{(i,j)}$ in a $\sing$-configuration in $P_{k-1}$, we say that $(i,j)_-$ is \textbf{ideal}. Otherwise, we say that $(i,j)_-$ is \textbf{non ideal}.
\end{itemize}    
\end{definition}
The idea of the algorithm is to first check whether $(i,j)_-$ is ideal or not. If it is not then we check if the level $j$ is stable or not. Finally, we check if $(i,j)_-$ is the smallest non ideal pair in $l_{(i,j)}$. If the level is unstable, we additionally check whether $(i,j)_-$ is the second smallest non ideal pair in $l_{(i,j)}$. In each one of the $6$ cases we obtain we can use one of the operations introduced previously. \\

\noindent
In the description of the cases, when we say that we use operation 2 with $(i,j')_+$ and $(i,j')_+$ is in $\hdoub$-configuration (Cases 3, 4 and 6), we are implicitly referring to the pair $(i',j')_-$ in $\nhdoub$-configuration with $(i,j')_+$ and we will prove (Lemmas \ref{Case3Possible}, \ref{Case4Possible} and \ref{Case6Possible}) that the pair $(i',j)_+$ in $\sing$-configuration, which ensures the we can use operation 2.
\begin{itemize}
    \item \textbf{Case 1.} If $(i,j)_-$ is ideal, then we use operation 1 with the maximal (with the order $ \rightarrow$) positive pair $(i+l,j)_+$, $i<i+l\leq j$, contained in the list $l_{(i,j)}$ that is in a $\sing$-configuration. 
    \item \textbf{Case 2.} If the level $j$ is unstable and $(i,j)_-$ is the smallest not ideal pair on  $l_{(i,j)}$, then we will use operation 3.
    \item \textbf{Case 3.}  If the level $j$ is unstable and $(i,j)_-$ is the second smallest not ideal pair on  $l_{(i,j)}$. We will prove (lemma \ref{Case3Possible}) that there exists a pair $(i,j')_+$ contained in a $\sing$-configuration or in a $\hdoub$-configuration. Then we use operation 1 in the first instance and use operation 2 in  the second instance. We record the pair $(i,j')_+$ to use it on Case 4 with the next non ideal pair. 
    \item \textbf{Case 4.}  If the level $j$ is unstable and $(i,j)_-$ is not ideal but neither the smallest or the second smallest not ideal pair on  $l_{(i,j)}$. Let $(i_1,j)_-$ be the second not ideal pair on the list $l_{(i,j)}$. From Case 3, $(i_1,j)_-$ is contained in a $\nvdoub$-configuration with some positive pair $(i_1,j_1)_+$. We will prove (lemma \ref{Case4Possible}) that the positive pair $(i,j_1)_+$ is contained in a $\sing$-configuration or in a $\hdoub$-configuration. Then we use operation 1 in the first instance and use operation 2 in the second instance.
    
    \item \textbf{Case 5.} If the level $j$ is stable and $(i,j)_-$ the smallest not ideal pair on  $l_{(i,j)}$. We will prove (lemma \ref{Case5Possible}) that there is a pair $(i',i-1)_-$ contained in a $\nsing$-configuration and the pair $(i',j)_+$ is contained in a $\sing$-configuration. Then we use operation 4. We record the pair $(i',i-1)_-$ to use it on Case 6 with the next non ideal pair. 
    
    \item \textbf{Case 6.}  If the level $j$ is stable and $(i,j)_-$ is not ideal but not the smallest not ideal pair on  $l_{(i,j)}$. Let $(i_1,j)_-$ be the smallest not ideal pair on the list $l_{(i,j)}$. From Case 5, $(i_1,j)_-$ is contained in a $\utrip$-configuration with some negative pair $(i',i_1-1)_-$. We will prove (lemma \ref{Case6Possible}) that the positive pair $(i,i_1-1)_+$ is contained in a $\sing$-configuration or in a $\hdoub$-configuration. Then we use operation 1 in the first instance and use operation 2 in  the second instance.
\end{itemize}

The number of $p \in P_k$ in points {\it{(4)}} or {\it{(5)}} only increases when we enter an unstable level for the first time (\ie Case 2) to use operation 3, however the number of unstable levels is precisely $\min(\alpha+1,\beta)$. Hence, it is enough to prove that the algorithm is correct, \ie to prove that the can actually perform the required operation on cases 1 through 6 of the algorithm. This is proven in the next subsection. With special attention to cases 3 through 6 that correspond to the lemmas \ref{Case3Possible}, \ref{Case4Possible}, \ref{Case5Possible}, and \ref{Case6Possible}.

We refer the reader to the appendix for a step-by-step description of the algorithm along with an example of a partition produced from it. 

\begin{figure}[h]
    \centering
    \vspace{-5mm}
    \begin{subfigure}[t]{0.45\textwidth}
    \begin{tikzpicture}
    \matrix [matrix of math nodes] (m)
        {
            (i,j)_{-} &\hdots &(i',j)_{-} & \hdots & (i+l,j)_{+} & \hdots & (i'+l',j)_{+}  \\
            & & & & & &  \\
            & & & & & &\\
        };  
        \draw[color=red][-] (m-1-1) edge[out=25,in=155] (m-1-5);
        \draw[color=red][-] (m-1-3) edge[out=25,in=155] (m-1-7);
\end{tikzpicture} \caption{Impossible configuration \textit{(1)}.}   
    \label{badconf1}
    \end{subfigure}
    \hfill \hfill
    \begin{subfigure}[t]{0.45\textwidth}
        \centering
    \begin{tikzpicture}
    \matrix [matrix of math nodes] (m)
        {
            (i,j)_{-} & \hdots &(i',j)_{-} & \hdots & (i+l,j)_{+}  \\
            &  & \vdots & &  \\
            &  & (i',j-l')_{+} & & \\
        };  
        \draw[color=red][-] (m-1-1) edge[out=25,in=155] (m-1-5);
        \draw[color=red][-] (m-3-3) edge[out=45,in=-45] (m-1-3);
\end{tikzpicture}    \caption{Impossible configuration \textit{(2)}.}
    \label{badconf2}
    \end{subfigure}
    \vspace{0.2cm}

    \centering
    \begin{subfigure}[t]{0.45\textwidth}
        \centering
     \begin{tikzpicture}
    \matrix [matrix of math nodes] (m)
        {
            &  &(i',j+l)_{-}\\
            &  & \vdots  \\
            (i,j)_-& \hdots  & (i',j)_{+}\\
        };  
        \draw[color=red][-] (m-3-3) edge[out=45,in=-45] (m-1-3);
\end{tikzpicture}    \caption{Impossible configuration \textit{(3)}.}
    \label{badconf3}
    \end{subfigure}
    \hfill
   \begin{subfigure}[t]{0.45\textwidth}
    \centering
    \begin{tikzpicture}
    \matrix [matrix of math nodes] (m)
        {
           (i,j)_- & \hdots & (i',j)_-\\
           \vdots & & \vdots\\
           (i,j-l_1)_+ &  & (i',j-l_2)_+\\
        };  
        \draw[color=red][-] (m-3-1) edge[out=135,in=225] (m-1-1);
        \draw[color=red][-] (m-3-3) edge[out=45,in=-45] (m-1-3);
\end{tikzpicture} \caption{Impossible configurations \textit{(4)} and \textit{(5)}.}   
    \label{badconf4}
    \end{subfigure}
    
    \caption{Configurations from lemma \ref{notpossible}.}
        \label{badconf}
\end{figure}

\subsection{Correctness of the algorithm}

We start by stating some technical lemmas that describe the structure of the partition constructed via our algorithm.  

\begin{lemma}\label{notpossible}
Let $0 \leq k \leq N$, suppose that we are able to construct $P_k$ following the method above. The following five configurations are impossible in $P_k$. See Figure \ref{badconf}.
\begin{enumerate}
    \item $(i,j)_-$ is in a $\nhdoub$-configuration with $(i+l,j)_+$, $(i',j)_-$ is in a $\nhdoub$-configuration with $(i'+l',j)_+$ and $i < i' <i+l < i'+l'$. 
    \item $(i,j)_-$ is in a $\nhdoub$-configuration with $(i+l,j)_+$, $(i',j)_-$ is in a $\nvdoub$-configuration with $(i',j-l')_+$ and $i < i' <i+l$.
    \item  $(i,j)_-$ is in a $\nvdoub$-configuration, $\dtrip$-configuration or $\utrip$-configuration, $(i',j)_+$ is in a $\vdoub$-configuration with $(i',j+l)_-$ and $i<i'$.
    \item $(i,j)_-$ is in a $\nvdoub$-configuration with $(i,j-l_1)_+$, $(i',j)_-$ is in a $\nvdoub$-configuration with $(i',j-l_2)_+$ $i<i'$ and $l_1\neq l_2$.
    \item $(i,j)_-$ is in a $\nvdoub$-configuration with $(i,j-l_1)_+$, $(i',j)_-$ is in a $\nvdoub$-configuration with $(i',j-l_2)_+$,  $i<i'$ and $(i,j)_-$ is the smallest non ideal negative pair on the list $l_{(i,j)}$.
\end{enumerate}
\end{lemma} 

\begin{proof}
All these configurations are impossible due to the order in which we perform the algorithm. The first three items are consequences of Case 1 of the construction given above, item \textit{(4)} is a consequence of Case 3 and item  \textit{(5)} is a consequence of Case 2.  
\end{proof}

We will postpone the proofs of the following five lemmas to section \ref{SectionTechnicalLemmas}. Lemmas \ref{Combinatorial1} and \ref{MinimalPairs} ensure the existence of some positive pairs, while lemmas \ref{ImpConf1}, \ref{ImpConf3} and \ref{Strat} further establish the rigid structure of the partition we construct. See figure \ref{MoreImp}. 

\begin{lemma}\label{Combinatorial1}
If $\vec{x}=(x_1,\dots,x_n) \in \big([-1,1]\backslash\{0\}\big)^n$, $n$ is odd and if the number of $x_i$ with $\sign(x_i)=-1$ is odd, then $b_{1}(\vec{x})+1=b_{-1}(\vec{x})$, where $$ b_{l}(\vec{x}):=|\{(i,j) \in K_{\vec{x}} : s(a_{\vec{x}(i,j)})=l, i=1 \text{ or } j=n \}|.$$
\end{lemma}

\begin{lemma}\label{MinimalPairs}
Let $h_{(i,j)}=\{(i+1,j),(i+2,j),\dots,(j,j)\}$, $N_{(i,j)}$ and $P_{(i,j)}$ be the number of negative pairs and positive pairs respectively. Then the following holds.
\begin{itemize}
    \item Suppose $(i,j)_-$ is the smallest non ideal negative pair on the list  $l_{(i,j)}$ ( \ie we are are in Cases 2 or 5), then $N_{(i,j)}=P_{(i,j)}$.
    \item Suppose $h_{(i,j)}$ contains more positive canonical pairs than negative canonical pairs, then $N_{(i,j)} \leq P_{(i,j)}$.
\end{itemize}
\end{lemma}

\begin{lemma}\label{ImpConf1}
Given $P_k$ a step in the construction of $\eta_{\vec{x}}$. We never have $(i,j)_-$ and $(i,j')_-$ both smallest non ideal pairs on the lists $l_{(i,j)}$ and $l_{(i,j')}$ respectively that are contained in the same column. 
\end{lemma}

\begin{figure}[h]
    \centering
    \vspace{-5mm}
    \begin{subfigure}[t]{0.45\textwidth}
    \begin{tikzpicture}
    \hspace{1.5cm}\matrix [matrix of math nodes] (m)
        {
            (i,j)_{-} & &  \\
            \vdots& & \\
            (i,j_1)_{+}&\hdots & (i',j_1)_{-} \\
        };   
        \draw[color=red][-] (m-3-1) edge[out=-45,in=225] (m-3-3);
\end{tikzpicture} \caption{Impossible configuration in lemma \ref{ImpConf3}.}   
    \end{subfigure}
    \hfill \hfill
    \begin{subfigure}[t]{0.45\textwidth}
        \centering
    \begin{tikzpicture}
    \matrix [matrix of math nodes] (m)
        {
           (i,j_2) &  & (i',j_2')\\
           \vdots & & \vdots\\
           (i,j) & \hdots & (i',j)\\
        };  
        \draw[color=red][-] (m-3-1) edge[out=135,in=225] (m-1-1);
        \draw[color=red][-] (m-3-3) edge[out=45,in=-45] (m-1-3);
\end{tikzpicture}    \caption{Impossible configuration in lemma \ref{Strat}.}
    \end{subfigure}
    \caption{Impossible configurations from lemmas \ref{ImpConf3} and \ref{Strat}.}
    \label{MoreImp}
\end{figure}

\begin{lemma}\label{ImpConf3}
Given $P_k$ a step in the construction of $\eta_{\vec{x}}$. Suppose either the level $j$ is unstable and $(i,j)_-$ is not ideal but not the smallest or second smallest non ideal on $l_{(i,j)}$ or that the level $j$ is stable and $(i,j)_-$ is not ideal but not the smallest non ideal on $l_{(i,j)}$ (\ie We are in case 4 or 6), then the positive positive pair $(i,j_1)_+$ cannot be in the same partition as $(i',j_1)_-$, $i<i'$.
\end{lemma}

\begin{definition}
Given $P_k$ a step in the construction of $\eta_{\vec{x}}$. We say that two horizontal lines $h_{j_1}=\{(1,j_1),(2,j_1),\dots,(j_1,j_1)\}$ and $h_{j_2}=\{(1,j_2),(2,j_2),\dots,(j_2,j_2)\}$ are connected if there exist an element $p\in P_k$ such that $p\cap h_{j_1} \neq \emptyset$ and $p\cap h_{j_2} \neq \emptyset$. 
\end{definition}

\begin{lemma}\label{Strat}
Given $P_k$ a step in the construction of $\eta_{\vec{x}}$. Given an a horizontal line $h_j=\{(1,j),(2,j),\dots,(j,j)\}$, there exist at most one $j_1 < j$ and one $j_2>j$ such that $h_j$ is connected to $h_{j_1}$ and to $h_{j_2}$.
\end{lemma}

We now prove the correctness of the algorithm via lemmas \ref{Case3Possible}, \ref{Case4Possible}, \ref{Case5Possible}, and \ref{Case6Possible}.

\begin{lemma}[Case 3]\label{Case3Possible}
Suppose $(i,j)_-$ is the second smallest non ideal pair on the list $l_{(i,j)}$ in a unstable level. Then there exist a positive pair $(i,j')_+$ contained in $\sing$-configuration (So we can apply operation 1), or in $\hdoub$-configuration with $(i',j')_-$ and $(i',j)_+$ is in $\sing$-configuration (So we can apply operation 2). 
\end{lemma}
\begin{proof}
Since  $(i,j)_-$ is the second smallest non ideal pair on the list $l_{(i,j)}$, then necessarily $l_{(i,j)}$ contains exactly two more negative pairs than positive pairs. Now applying lemma \ref{Combinatorial1} to the vector $(x_i,\dots,x_j)$ ensures that on the vertical list $v_{(i,j)}=\{(i,i),\dots,(i,j-1),(i,j-1)\}\cap K_{\vec{x}}$ there is exactly one more positive pair than negative pairs. Since all pairs on $v_{(i,j)}$ are contained in a partition, definition \ref{cases} ensures that there exists some $(i,j')_+ \in v_{(i,j)}$ contained either in $\sing$-configuration or in $\hdoub$-configuration with $(i',j')_-$. In the
first situation the proof is over, in the second situation we still need to verify that $(i',j)_+$ is in $\sing$-configuration. Let's consider all possible configurations

We have that $(i',j)_+$  cannot be in $\hdoub$, $\vdoub$ or $\iquad$-configuration due to the order in which the construction is performed. It is not in $\trip$-configuration, since we are at an unstable level. It is not in $\tquad$-configuration, since $(i,j)_-$ is the second smallest non ideal pair on $l_{(i,j)}$. The only option left is that $(i',j')_+$ is in $\sing$-configuration.
\end{proof}

\begin{figure}[h]
    \centering
    \vspace{-5mm}
    \begin{subfigure}[t]{0.45\textwidth}
    \begin{tikzpicture}
    \hspace{1.5cm}\matrix [matrix of math nodes] (m)
        {
            (i',j)_+ & \hdots  & (i,j)_{-}   \\
            \vdots& & \vdots\\
            (i',j')_{-}&\hdots & (i,j')_{+} \\
        };   
        \draw[color=red][-] (m-1-1) edge[out=45,in=135] (m-1-3);
        \draw[color=red][-] (m-3-1) edge[out=-45,in=225] (m-3-3);
        \draw[color=red][-] (m-3-1) edge[out=135,in=225] (m-1-1);
        \draw[color=red][-] (m-3-3) edge[out=45,in=-45] (m-1-3);
\end{tikzpicture} \caption{Illustration for lemma \ref{Case3Possible}.}   
    \end{subfigure}
    \hfill \hfill
    \begin{subfigure}[t]{0.45\textwidth}
        \centering
    \begin{tikzpicture}
    \matrix [matrix of math nodes] (m)
        {
            (i',j)_+ & \hdots  & (i,j)_{-} & \hdots & (i_1,j)_-  \\
            \vdots& & \vdots & & \vdots \\
            (i',j_1)_{-}&\hdots & (i,j_1)_{+} & \hdots & (i_1,j_1)_+ \\
        };    
        \draw[color=red][-] (m-1-1) edge[out=45,in=135] (m-1-3);
        \draw[color=red][-] (m-3-1) edge[out=-45,in=225] (m-3-3);
        \draw[color=red][-] (m-3-1) edge[out=135,in=225] (m-1-1);
        \draw[color=red][-] (m-3-3) edge[out=45,in=-45] (m-1-3);
        \draw[color=red][-] (m-3-5) edge[out=45,in=-45] (m-1-5);
\end{tikzpicture}    \caption{Illustration for lemma \ref{Case4Possible}.}
    \end{subfigure}
    \caption{Illustrations for lemmas \ref{Case3Possible} and \ref{Case4Possible} using operation 2.}
\end{figure}

\begin{lemma}[Case 4]\label{Case4Possible}
Suppose $(i,j)_-$ is a non ideal pair on the list $l_{(i,j)}$ in a unstable level, but not the smallest or second smallest on $l_{(i,j)}$. Let $(i_1,j)_-$ be the second smallest non ideal pair on the list $l_{(i,j)}$. From case 3, $(i_1,j)_-$ is in $\nvdoub$ with some pair $(i_1,j_1)_+$. 
Then $(i,j_1)_+$ is either in $\sing$-configuration (So we can apply operation 1) or in $\hdoub$-configuration with some pair $(i',j_1)_-$ and $(i',j)_+$ is in $\sing$-configuration (So we can apply operation 2).
\end{lemma}
\begin{proof}
We first verify that $(i,j_1)_+$ is either in $\sing$-configuration or in $\hdoub$-configuration. lemma \ref{Strat} implies that $(i,j_1)_+$ cannot be in $\vdoub$ or $\iquad$-configuration while lemma \ref{ImpConf3} implies it cannot be in a $\tquad$ o $\trip$-configuration.

Let's prove that if it is $\hdoub$-configuration with some pair $(i',j_1)_-$ then $(i',j)_+$ is in $\sing$-configuration. We consider all possible cases. It cannot be in $\hdoub$, $\vdoub$, or $\iquad$-configuration due to the order in which the construction is performed.  It is not in $\trip$-configuration since we are at an unstable level. lemma \ref{Strat} implies that it is not in $\tquad$-configuration.
\end{proof}

\begin{lemma}[Case 5]\label{Case5Possible}
Suppose $(i,j)_-$ is the smallest non ideal pair on the list $l_{(i,j)}$ in a stable level. Then there is a pair $(i',i-1)_-$ contained in a $\nsing$-configuration and the pair $(i',j)_+$ is contained in a $\sing$-configuration (So that we can apply operation 4).
\end{lemma}
\begin{proof}
For each $1 \leq v \leq n$ and let $\alpha'_v$ be the number positive pairs on the sequence $\{(1,i)\}_{i=1}^{v}$ and let $\beta'_v=v-\alpha'_v$ the number of negative pairs. Similarly let $(\alpha_{v},\beta_{v})$ be the number of respectively positive and negative pairs in the line $\{(1,v),\dots,(v,v)\}$ and let $p_v$ and $m_v$ be the number of positive and negative variables $y_1,\dots,y_{v+1}$. It follows from the change of variables that $\min(\alpha'_v+1,\beta'_v)=\min(p_v,m_v)=\min(\alpha_v+1,\beta_v)$. It is then enough to study the couples $(\alpha_v,\beta_v)$ to determine which levels are stable and which are not. 
\vspace{-5mm}
\begin{figure}[h!] 
    \centering
    \begin{tikzpicture}
        \matrix [matrix of math nodes] (m)
        {
          &   & \color{red}{I_2} &   &   &   & \color{red}{I_1}  &    \\
        (1,j) & \hdots &\hdots & (i-1,j) & (i,j)_- & \hdots & (j-1,j)  & (j,j)_+   \\
        (1,j-1) & & & & \vdots &   &  (j-1,j-1)   & \\
        \vdots & & & & \vdots  & \iddots &    & \\
        (1,i) &\hdots & (i-2,i) &  (i-1,i)  & (i,i) &    &    & \\
        (1,i-1) &\hdots & (i-2,i-1) & (i-1,i-1)  & &    &    & \\
        (1,i-2) &\hdots & (i-2,i-2) &  \color{red}{I_3} &  &    &    & \\
        };  
        \draw[color=red] (m-6-1.south west) rectangle (m-6-4.north east);
        \draw[color=red] (m-2-1.south west) rectangle (m-2-4.north east);
        \draw[color=red] (m-2-6.south west) rectangle (m-2-8.north east);
\end{tikzpicture}
    \caption{Illustration of lemma \ref{Case5Possible}} \label{IlustrationCase5}
\end{figure}

We first compute $(\alpha_{i-2},\beta_{i-2})$, $(\alpha_{i-1},\beta_{i-1})$, $(\alpha_{j-1},\beta_{j-1})$ and $(\alpha_j,\beta_j)$. We then show that the level $i-1$ is unstable and that in contains a non ideal canonical pair in $\nsing$-configuration. We finally show that it is possible to apply operation 4. Our calculation depends on the sign of  $(i-1,i-1)$, we provide a proof for the positive case, the negative case being analogous. See figure \ref{IlustrationCase5}.

\begin{enumerate}
    \item Suppose that $(i-1,i-1)_+$ is positive. Then $(\alpha_{i-1},\beta_{i-1})=(\alpha_{i-2}+1,\beta_{i-2})$. Since $(i,j)_-$ is the smallest non ideal pair on this level, lemma \ref{MinimalPairs} implies that in the segment $I_1=\{(i+1,j),\dots,(j,j)\}$ there is the same number $r$ of negative and positive pairs. Since $(i,j)_-$ is negative, By comparing the segment $I_2=\{(1,j),\dots,(i-1,j)\}$ with the segment $I_3=\{(1,i-1),\dots,(i-1,i-1)\}$,we conclude that $I_2$ contains exactly $\beta_{i-1}$ positive pairs and $\alpha_{i-1}$ negative pairs. We conclude that $(\alpha_j,\beta_j)=(\beta_{i-2}+r,\alpha_{i-2}+r+2)$.  Finally, if $i=j$ then $(\alpha_{j-1},\beta_{j-1})=(\alpha_{i-2}+1,\beta_{i-2})$, if $i<j$, then necessarily $(j,j)_+$ is a positive pair. If it was negative it would contradict the fact that $(i,j)_-$ is the smallest non ideal on this level. Since $(j,j)_+$ we know that the sign of each $(i'',j-1)$ is the same as the sign of $(i'',j)$ for each $1\leq i'' \leq j-1$, hence $(\alpha_{j-1},\beta_{j-1})=(\alpha_{j}-1,\beta_{j})=(\beta_{i-2}+r-1,\alpha_{i-2}+r+2)$.
    \item Hence, the stability of level $j$ is translated into the identity $\min(\alpha_{i-2}+2,\beta_{i-2})=\min(\alpha_{i-2}+2,\beta_{i-2}+1)$. This last identity implies that $\beta_{i-2}\geq \alpha_{i-2}+2$. Hence $\min(\alpha_{i-2}+1,\beta_{i-2})=\alpha_{i-2}+1<\alpha_{i-2}+2=\min(\alpha_{i-2}+2,\beta_{i-2})$, hence the level $i-1$ is unstable. 
    Moreover, if the number of positive canonical pairs on $\{(1,i-2),\dots,(i-2,i-2)\}$ is greater than or equal to the number of negative canonical pairs, lemma \ref{MinimalPairs} ensures that $\alpha_{i-2}+1\geq \beta_{i-2}$, which contradicts $\beta_{i-2}\geq \alpha_{i-2}+2$. This proves that there is strictly more negative pairs than positive pairs on $\{(1,i-2),\dots,(i-2,i-2)\}$, in particular there must exist a negative canonical pair $(i',i-1)_-$ for which we applied operation 3. lemma \ref{ImpConf1} ensures that $(i',i-1)_-$ is not contained in $\dtrip$-configuration and hence $(i',i-1)_-$ must be contained in $\nsing$-configuration. 
    \item Since the pairs $(i',i-1)_-$ and $(i,j)_-$ are negative, the pair $(i',j)_+$ must be positive. We finally need to verify that it is contained in  $\sing$-configuration. It cannot be in $\hdoub$, $\vdoub$, or $\iquad$-configuration due to the order in which the construction is performed.  It is not in $\trip$-configuration or in a  $\tquad$-configuration since $(i,j)_-$ is the smallest non ideal pair on $l_{(i,j)}$. \qedhere
\end{enumerate}
\end{proof}

\begin{lemma}[Case 6]\label{Case6Possible}
Suppose $(i,j)_-$ is a non ideal pair on the list $l_{(i,j)}$ in a stable level, but not the smallest on $l_{(i,j)}$. Let $(i_1,j)_-$ be the smallest non ideal pair on the list $l_{(i,j)}$. From case 5, $(i_1,j)_-$ is contained in a $\utrip$-configuration with some negative pair $(i',i_1-1)_-$. Then $(i,i_1-1)_+$ is either in $\sing$-configuration (So we can apply operation 1) or in $\hdoub$-configuration with some pair $(i',i_1-1)_-$ and $(i',j)_+$ is in $\sing$-configuration (So we can apply operation 2).
\end{lemma}
\begin{proof}
The proof is analog to the proof of lemma \ref{Case4Possible}.
\end{proof}

\section{Proofs of technical lemmas}\label{SectionTechnicalLemmas}

In this section we complete the proofs of the technical lemmas describing the structure of the partition. 

\begin{figure}[h!] 
    \centering
    \vspace{-4mm}
    \begin{tikzpicture}
        \matrix [matrix of math nodes] (m)
        {
            a_{\vec{x}'(1,n+2)} &a_{\vec{x}'(2,n+2)} &\hdots &a_{\vec{x}'(k,n+2)} &a_{\vec{x}'(k+1,n+2)} &a_{\vec{x}'(k+2,n+2)} &\hdots  &a_{\vec{x}'(n+2,n+2)}  \\
            \vdots& \vdots & &\vdots  &\vdots &\vdots & \iddots & \\
            a_{\vec{x}'(1,k+2)} &a_{\vec{x}'(2,k+2)} &\hdots &a_{\vec{x}'(k,k+2)} &a_{\vec{x}'(k+1,k+2)} & e_2 & &  \\    
            a_{\vec{x}'(1,k+1)} &a_{\vec{x}'(2,k+1)} &\hdots &a_{\vec{x}'(k,k+1)} & e_1 & & & \\      
            a_{\vec{x}'(1,k)} &a_{\vec{x}'(2,k)} &\hdots &a_{\vec{x}'(k,k)} & & & & \\     
            \vdots &\vdots &\iddots & & & & & \\    
            a_{\vec{x}'(1,2)} &a_{\vec{x}'(2,2)} & & & & & & \\     
            a_{\vec{x}'(1,1)} & & & & & & & \\      
        };  
        \draw[color=red] (m-4-1.south west) rectangle (m-3-1.north east);
        \draw[color=red] (m-1-5.south west) rectangle (m-1-6.north east);
\end{tikzpicture}
    \vspace{-2mm}
    \caption{Elements for which we must inspect the product sign.}
    \label{figure3}
\end{figure}
\begin{proof}[Proof of lemma \ref{Combinatorial1}]
Note that we just need to inspect the vertical and horizontal edges of the triangle in Figure \ref{figure3}. We proceed by induction on $n$, the case $n=1$ being trivial. If  the sequence of signs of the $x_k$ is $(-,+,-,\ldots,+,-)$, in which case necessarily $n\equiv1$ (mod 4), the lemma is a straight-forward calculation that does not even require the inductive assumption. For example, in the case $n=1$, $b_{-1}(\vec{x})=1$ and $b_{1}(\vec{x})=0$. If $\vec{x}$ does not have this sign pattern, then the sequence of signs must necessarily contain consecutive $++$ or $--$.  The inductive step amounts to the following 
\vskip0.3cm
\noindent {\it{Claim}}. If $b_{1}(\vec{x})+1=b_{-1}(\vec{x})$ for some $\vec{x} \in \big([-1,1]\backslash\{0\}\big)^n$, then $b_{1}(\vec{x}')=b_{-1}(\vec{x}')$ for $\vec{x}':=(x_1,\dots,x_k,e_1$ $,e_2,x_{k+1},\dots,x_n) \in \big([-1,1]\backslash\{0\}\big)^{n+2}$ where $\sign(e_1)=\sign(e_2)$ and $0 \leq k \leq n$.

To prove this, note that if $j\leq k$ then $a_{\vec{x}'(i,j)}=a_{\vec{x}(i,j)}$. Also, if $i \geq k+3$ then $a_{\vec{x}'(i,j)}=a_{\vec{x}(i-2,j-2)}$. Thus in both cases $s(a_{\vec{x}'(i,j)})=s(a_{\vec{x}(i,j)})$ and $s(a_{\vec{x}'(i,j)})=s(a_{\vec{x}(i-2,j-2)})$. Moreover, if $i \leq k$ and $j \geq k+3$, then $s(a_{\vec{x}'(i,j)})= s(a_{\vec{x}(i,j-2)})$. So we just need to inspect the elements $a_{\vec{x}'(1,k+1)}$, $a_{\vec{x}'(1,k+2)}$, $a_{\vec{x}'(k+1,n+2)}$ and $a_{\vec{x}'(k+2,n+2)}$.

If $\sign(e_1)=1$, then $s(a_{\vec{x}'(1,k+1)})= s(a_{\vec{x}'(1,k+2)})$ and therefore $a_{\vec{x}'(1,k+1)} \in K_{\vec{x}}$ if and only if $a_{\vec{x}'(1,k+2)} \not \in K_{\vec{x}}$. The same statement holds for $a_{\vec{x}'(k+1,n+2)}$ and $a_{\vec{x}'(k+2,n+2)}$. Also, since the number of $\{x_i\}_{i=1}^n$ with $\sign(x_i)=-1$ is odd, then $\sign(\prod_{i=1}^k x_i)=-\sign(\prod_{i=k+1}^{n} x_i)$. This implies that $\s(a_{\vec{x}'(1,k+1)}) \neq s(a_{\vec{x}'(k+1,n+2)})$, so we have $b_{1}(\vec{x}')+1=b_{-1}(\vec{x}')$. 

If $\sign(e_1)=-1$, then $\s(a_{\vec{x}'(1,k+1)}) \neq  \s(a_{\vec{x}'(1,k+2)})$ and therefore $a_{\vec{x}'(1,k+1)} \in K_{\vec{x}}$ if and only if $a_{\vec{x}'(1,k+2)} \in K_{\vec{x}}$. The same statement holds for $a_{\vec{x}'(k+1,n+2)}$ and $a_{\vec{x}'(k+2,n+2)}$. So we have $b_{1}(\vec{x}')+1=b_{-1}(\vec{x}')$.
\end{proof}
\begin{proof}[Proof of lemma \ref{MinimalPairs}]
The second statement proof is contained in the first statement proof. 

\noindent {\it{Claim 1}}. The number of positive canonical pairs in $h_{(i,j)}$ is equal to the number of negative canonical pairs in $h_{(i,j)}$.\\
\textit{Proof of claim 1:} If the number of positive canonical pairs is strictly larger, then $(i,j)_-$ would be ideal. If the number of negative canonical pairs is strictly greater, then $(i,j)_-$ wouldn't be a smallest non ideal $l_{(i,j)}$ .\\
\noindent {\it{Claim 2}}. Given any two adjacent pairs, Either both are canonical of opposite sign, non-canonical of opposite sign, or have the same sign but only one is canonical. \\
\textit{Proof of claim 2:} This can be verified from the definition of canonical and non-canonical. 

\begin{figure}[H] 
    \centering
    \vspace{-2mm}
    \begin{tikzpicture}
        \matrix [matrix of math nodes] (m)
        {
            (i,j)_- & (i+1,j) &\dots & (i_r-1,j) & (i_r,j) & \dots & (i_1,j) & (i_1+1,j) & \dots & (j,j)   \\
              &   & \textcolor{red}{I_r} &   &   &   &   &   & \textcolor{red}{I_0} &     \\
        };  
        \draw[color=red] (m-1-2.south west) rectangle (m-1-4.north east);
        \draw[color=red] (m-1-8.south west) rectangle (m-1-10.north east);
\end{tikzpicture}
    \vspace{-5mm}
    \caption{Sequence of intervals $\{I_s\}_{s=0}^r$.}
    \label{figurerepcounting}
\end{figure}

We need to verify now that the number of negative non-canonical pairs is equal to the number of positive non-canonical pairs. To this end, let $i_1,i_2,\dots,i_r$ be the coordinates corresponding to the canonical pairs on $h_{(i,j)}$, fix $i_{r+1}=i$. We study the interval of consecutive non-canonical pairs defined by $I_s=\{(i_{s+1}+1,j),\dots,(i_{s}-1,j)\}$ for $s=1,\dots,r$ and $I_0=\{(i_1+1,j),\dots,(j,j)\}$. Remark that these may be empty. See figure \ref{figurerepcounting} for a representation. 

We want now to count the number of positive and negative pairs on each interval $I_s$. There are three cases to consider.

\begin{itemize}
    \item If $(i_{s+1},j)$ and $(i_{s},j)$ have opposite sign, then the interval $I_s$ has even length and claim 2 implies that it contains the same number of positive pairs and negative pairs.
    \item If  $(i_{s+1},j)$ and $(i_{s},j)$ both are positive, then $I_s$ has odd length and claim 2 implies that it contains exactly one more positive pair than negative pairs. 
    \item If  $(i_{s+1},j)$ and $(i_{s},j)$ both are negative,  then $I_s$ has odd length and  claim 2 implies that it contains exactly one more negative pair than positive pairs. 
\end{itemize}

In addition, $I_0$ has always even length and then contains the same number of positive pairs and negative pairs. We can separate the intervals $I_s$ in two groups. When contained between two positive canonical pairs then the interval has a surplus of exactly one positive non-canonical pair. When contained between two negative canonical pairs then the interval has a surplus of exactly one negative non-canonical pair. Finally, claim 1 implies that the number of these pairs is the same. We conclude that $N_{(i,j)}=P_{(i,j)}$.
\end{proof}

\begin{proof}[Proof of lemma \ref{ImpConf1}]
    Let these smallest non ideal pairs be $(i,j)_-$ and $(i,j')_-$ with $j'<j$. As in the proof of lemma \ref{MinimalPairs}, in the segment $I_1=\{(i+1,j'),\dots(j',j')\}\cap K_v$ we have the same number of positive and negative canonical pairs. This implies that in the segment $I_2=\{(i+1,j),\dots,(j',j)\}\cap K_v$ we also have the same number of positive and negative canonical pairs. 
\begin{figure}[h!] 
    \centering
    \vspace{-6mm}
    \begin{tikzpicture}
        \matrix [matrix of math nodes] (m)
        {
          &   & \color{red}{I_2} &   &   &   &    &    \\
        (i,j)_- & (i+1,j) &\hdots & (j',j) & (j'+1,j)_+ & \hdots & (j,j)   \\
        \vdots &           &      &        &  \vdots & \iddots\\
        (i,j'+1) &     \hdots      & \hdots     &   \hdots     &  (j'+1,j'+1) \\
        (i,j')_- & (i+1,j') & \hdots & (j',j') &  \color{red}{I_1} & &    & \\
        };  
        \draw[color=red] (m-2-2.south west) rectangle (m-2-4.north east);
        \draw[color=red] (m-5-2.south west) rectangle (m-5-4.north east);
\end{tikzpicture}
    \caption{Illustration of lemma \ref{ImpConf1}} \label{Helper1}
\end{figure}
    
    Suppose $(j'+1,j'+1)$ is negative. Then it is canonical, necessarily $(i,j'+1)_+$ is positive and finally using that $(i,j'+1)_+$ is canonical positive, that $(j'+1,j'+1)_-$ and $(i,j)_-$ are canonical negatives we conclude that $(j'+1,j)_+$ is canonical positive. 
    
    Suppose that  $(j'+1,j'+1)$ is positive. Then it is non-canonical and we know from the proof of lemma \ref{MinimalPairs} that $(j',j')$ is positive non-canonical, and hence $(j',j'+1)_+$ is positive canonical. Since $(j',j')$ is positive non-canonical, $(j'+1,j'+1)_-$ and $(i,j)_-$ are negatives canonical, then $(j',j)_+$ is positive non-canonical. Finally, since $(j',j'+1)_+$ is positive canonical and $(j',j)_+$, $(j'+1,j'+1)_+$ are positive non-canonical, we conclude that $(j'+1,j)_+$ is canonical positive.
    
    In both cases, we have that $(j',j)_+$ is canonical positive. Hence, in the extended segment $I_2\cup\{(j'+1,j)\}$ the number of positive canonical pairs is strictly greater than the number of negative canonical pairs. In particular there must be a pair in $\sing$-configuration in the extended segment, this contradicts the fact that $(i,j)_-$ is the smallest non ideal on $l_{(i,j)}$.
\end{proof}

\begin{lemma}\label{ImpConf2}
Given $P_k$ a step in the construction of $\eta_{\vec{x}}$. Suppose either the level $j$ is unstable and $(i,j)_-$ is the second smallest non ideal on $l_{(i,j)}$ or that the level $j$ is stable and $(i,j)_-$ is the smallest non ideal on $l_{(i,j)}$ (\ie We are in case 3 or 5), then there cannot exist a negative pair $(i,j')_-$ in $\nhdoub$-configuration with $j'<j$.
\end{lemma}
\begin{proof}[Proof of lemma \ref{ImpConf2}]
Suppose that such a negative pair exists and $(i,j')_-$ is in $\nhdoub$-configuration with $(i',j')_+$, $i<i'$. Then $(i',j)_+$ is positive canonical. Because $(i,j)_-$ is the smallest (or second smallest) non ideal pair, $(i',j)_+$ is necessarily in $\hdoub$-configuration with some  negative pair $(i'',j)_-$ with $i<i''<i'$ (see figure \ref{Helper2}).

Since $(i,j')_-$ is in $\nhdoub$-configuration with $(i',j')_+$, this partition must have been constructed by means of Case 1 of the construction. In particular in the segment $I_1=\{(i+1,j'),(i+2,j'),\dots,(i'-1,j')\}\cap K_v$ the number of positive pairs is greater than or equal to the number of negatives pairs. This implies that the number of positive pairs is greater than or equal to the number of negatives in the segment $I_2=\{(i+1,j),(i+2,j),\dots,(i'-1,j)\}\cap K_v$. Finally, the negative pair $(i'',j)_-$ is contained in a partition with the positive pair $(i',j)_+$ not contained in $I_2$. Hence, there either exists a positive pair in $\sing$-configuration in the segment $I_2$, which contradicts the fact that $(i,j)_-$ is non ideal, or there exist a pair in $\trip$ or $\tquad$-configuration on $I_2$, which contradicts the fact that $(i,j)_-$ is the smallest (or second smallest) non ideal on $l_{(i,j)}$.
\end{proof}

\begin{figure}[h!] 
    \centering
    \vspace{-7mm}
    \begin{tikzpicture}
        \matrix [matrix of math nodes] (m)
        {
          &   & &  \color{red}{I_2}   &   &   &    &    \\
        (i,j)_- & (i+1,j) & \hdots & (i'',j)_- & \hdots & (i'-1,j) & (i',j)_+     \\
        \vdots &           &  & &    &        &  \vdots &  \\
        (i,j')_- & (i+1,j') &     \hdots      & \hdots     &   \hdots & (i'-1,j')     &  (i',j')_+ \\
          &   &  &   \color{red}{I_1}  &  & &    & \\
        };  
        \draw[color=red] (m-2-2.south west) rectangle (m-2-6.north east);
        \draw[color=red] (m-4-2.south west) rectangle (m-4-6.north east);
\end{tikzpicture}
    \vspace{-5mm}
    \caption{Illustration of lemma \ref{ImpConf2}} \label{Helper2}
\end{figure}

\begin{proof}[Proof of lemma \ref{ImpConf3}]

Suppose that the positive positive pair $(i,j_1)_+$ is in the same partition as $(i',j_1)_-$, $i<i'$, this can happens in two possible ways, either $(i',j_1)_-$ is in $\utrip$ or $\nuquad$-configuration with $(i,j')_-$. In both cases we have that $i'-1<j'$ so we can consider the segment $I_1=\{(i+1,j'),(i+2,j'),\dots,(i'-1,j')\}\cap K_v$, On $I_1$ the number of positive pairs is greater than or equal to the number of negative pairs, hence we have the same property in $I_2=\{(i+1,j),(i+2,j),\dots,(i'-1,j)\}\cap K_v$. Notice that the positive pair $(i',j)_+$ cannot be in $\trip$ or $\tquad$ configuration, if it was, by lemma \ref{notpossible} (4) and (5), it would need to be on the same partition as $(i',j_1)$, which is not possible. It cannot be in $\sing$-configuration, since it contradicts the fact that $(i,j)_-$ is not ideal. We conclude that it must be in $\hdoub$-configuration with some negative pair $(i'',j)_-$ in the segment $I_2$.

\begin{figure}[h!] 
    \centering
    \vspace{-5mm}
    \begin{tikzpicture}
        \matrix [matrix of math nodes] (m)
        {
          &   & &  \color{red}{I_2}   &   &   &    &    \\
        (i,j)_- & (i+1,j) & \hdots & (i'',j)_- & \hdots & (i'-1,j) & (i',j)_+     \\
        \vdots &           &  & &    &        &  \vdots &  \\
        (i,j_1)_+ &  &        &     &   &      &  (i',j_1)_- \\
        \vdots &           &  & \color{red}{I_1} &    &        &   &  \\
        (i,j')_- & (i+1,j') & \hdots & \hdots & \hdots & (i'-1,j') &\\ 
        };  
        \draw[color=red] (m-2-2.south west) rectangle (m-2-6.north east);
        \draw[color=red] (m-6-2.south west) rectangle (m-6-6.north east);
\end{tikzpicture}
    \vspace{-3mm}
    \caption{Illustration of lemma \ref{ImpConf3}} \label{Helper3}
\end{figure}

The rest of the proof is similar to lemma \ref{ImpConf2}. There either exists a positive pair in $\sing$-configuration in the segment $I_2$, which contradicts the fact that $(i,j)_-$ is non ideal, or there exist a positive pair in the segment $I_2$ contained in the same partition with another negative pair $(i''',j)_-$, $i'<i'''$. Note that $(i''',j)_-$ is in $\utrip$ or $\nuquad$-configuration, in the first case we obtain a contradiction since  $(i',j_1)_-$ is in another $\nuquad$-configuration or in a $\utrip$-configuration, in the second case we obtain a contradiction with the fact that $(i',j_1)_-$ is non ideal. 
\end{proof}

\begin{proof}[Proof of lemma \ref{Strat}]

It is clear from the construction of the partition in Cases 4 and 6, that there exists at most one $j_1<j$ such that $h_j$ is connected to $h_{j_1}$ (see lemma \ref{notpossible} (4) and (5)). We prove the other claim by contradiction. Suppose we there exist $h_{j_2}$ and $h_{j'_2}$ connected to $h_j$ with $j_2>j_2'>j$. We are necessarily in one of the following four cases. We reduce each case to constructing a specific element and proving that this element cannot exist by considering the possible partitions it is contained in. See figure \ref{Caseslemma25} with a red box indicating the specific element we study. For cases 3 and 4, the graphical representation changes depending on the comparison between $i_3$ and $i$, but the proof remains the same. 

\begin{figure}[h]
    \centering
    \vspace{-4mm}
    \begin{subfigure}[t]{0.45\textwidth}
    \begin{tikzpicture}
    \hspace{1.5cm}
    \matrix [matrix of math nodes] (m)
        {
            & & (i,j_2)_-\\
            & &  \vdots\\
            (i',j_2')_-& \hdots & (i,j_2')_-\\
            \vdots& & \vdots\\
           (i',j)_+ & \hdots & (i,j)_+\\
        }; 
         \draw[color=red] (m-3-3.south west) rectangle (m-3-3.north east);
\end{tikzpicture} \caption{Case 2, $i'<i$, of lemma \ref{Strat}.}   
    \end{subfigure}
    \hfill \hfill
    \begin{subfigure}[t]{0.45\textwidth}
        \centering
    \begin{tikzpicture}
    \matrix [matrix of math nodes] (m)
        {
            (i,j_2)_-& \hdots & (i',j_2)_-\\
            \vdots& & \vdots \\
             \vdots &  &(i',j_2')_-\\
            \vdots& & \vdots\\
            (i,j)_+& \hdots & (i',j)_+\\
        };
        \draw[color=red] (m-1-3.south west) rectangle (m-1-3.north east);
\end{tikzpicture}    \caption{Case 2, $i<i'$, of lemma \ref{Strat}.}
    \end{subfigure}
    \vspace{0.2cm}

    \centering
    \begin{subfigure}[t]{0.45\textwidth}
        \centering
     \begin{tikzpicture}
    \matrix [matrix of math nodes] (m)
        {
            & & (i,j_2)_-\\
            & &  \vdots\\
            (i_3,j_2')_+& \hdots & (i,j_2')_-\\
            \vdots& & \vdots\\
           (i_3',j)_- & \hdots & (i,j)_+\\
        }; 
        \draw[color=red] (m-3-3.south west) rectangle (m-3-3.north east);
\end{tikzpicture}    \caption{Case 3, $i_3<i$, of lemma \ref{Strat}.}
    \end{subfigure}
    \hfill
   \begin{subfigure}[t]{0.45\textwidth}
    \centering
    \begin{tikzpicture}
    \matrix [matrix of math nodes] (m)
        {
            & & (i_3,j_2)_+ &\hdots & (i,j_2)_- \\
            & &  \vdots & & \vdots \\
            (i',j_2')_- & \hdots & (i_3,j_2')_+& \hdots & (i,j_2')_-\\
            \vdots& & \vdots & & \\
           (i',j)_+ & \hdots & (i_3,j)_- & & \\
        }; 
        \draw[color=red] (m-3-5.south west) rectangle (m-3-5.north east);
\end{tikzpicture} \caption{Case 4,  $i_3<i$, of lemma \ref{Strat}.}   
    \end{subfigure}
    
    \caption{Cases of lemma \ref{Strat}.} \label{Caseslemma25}
\end{figure}

\begin{enumerate}
    \item {\bf{Levels $j_2$ and $j_2'$ are stable.}} Suppose that $(i,j_2)_-$ is the smallest non ideal pair in $h_{j_2}$, it should be connected to a negative pair $(i_3,i-1)_-$ in $h_j$. If $(i',j_2')_-$ is the smallest non ideal pair in $h_{j_2'}$, it should be connected to $(i_3',i'-1)_-$ in $h_j$. We obtain $i-1=j=i'-1$. We conclude that $(i,j_2)_-$ and  $(i',j_2')_-$ are both the smallest non ideal pair on their respective levels and contained in the same column, this contradicts lemma \ref{ImpConf1}.
    
    \item {\bf{Levels $j_2$ and $j_2'$ are unstable.}} Suppose that $(i,j_2)_-$ is the second smallest non ideal pair in $h_{j_2}$, it should be connected to $(i,j)_+$. If $(i',j_2')_-$ is the second smallest non ideal pair in $h_{j_2'}$, it should be connected to $(i',j)_+$. 
    
    We work two cases. If $i>i'$, then consider the negative canonical pair $(i,j_2')_-$. Since $(i',j_2')_-$ is the second smallest non ideal pair on $h_{j_2'}$, we have that $(i,j_2')_-$ cannot be in $\nvdoub$ or $\nuquad$-configuration. It cannot be in $\utrip$-configuration, since we are at an unstable level. It cannot be in a $\dtrip$ or $\nsing$-configuration, since it contradicts lemma \ref{ImpConf1}. It cannot be in $\ndquad$ or $\nhdoub$-configuration, since it contradicts lemma \ref{ImpConf2}. If $i<i'$, we repeat the previous argument by considering the possible configurations for $(i',j_2)_-$. 
    \item {\bf{Level $j_2$ is unstable and level $j_2'$ is stable.}} Suppose that $(i,j_2)_-$ the second smallest non ideal pair in $h_{j_2}$, it should be connected to $(i,j)_+$. If $(i',j_2')_-$ is the smallest non ideal pair in $h_{j_2'}$, it should be connected to $(i_3',j)_-$. By considering all possible partitions that include $(i,j_2')_-$ we obtain a contradiction. 
    \item {\bf{Level $j_2$ is stable and level $j_2'$ is unstable.}} Suppose that $(i,j_2)_-$ is the smallest non ideal pair in $h_{j_2}$, it should be connected to a negative pair $(i_3,j)_-$. If $(i',j_2')_-$ the second smallest non ideal pair in $h_{j_2'}$, it should be connected to $(i',j)_+$.  Since $i-1=j<j_2$, the pair $(i,j_2')_-$ exists, by considering all possible partitions that include it we obtain a contradiction. \qedhere
\end{enumerate} 
\end{proof}

\section{Further Remarks}

\begin{lemma}\label{inductive}
Let $\displaystyle P_n(y_1,\dots,\hat{y}_k,\dots,y_n)=\prod_{\substack{i<j\\ j,j\neq k}}\big(1-\dfrac{y_i}{y_{j}}\big)$, then
$$P_n(y_1,\dots,y_n)^{n-2}=\prod_{k=1}^n P_n(y_1,\dots,\hat{y}_k,\dots,y_n).$$
\end{lemma}
\begin{proof}
Note that each term $(1-y_i/y_j)$ appears $n-2$ times at each side of the equation.
\end{proof}

An iteration of lemma \ref{inductive} gives us the relation
$$P_n(y_1,\dots,y_n)^{\frac{(n-2)(n-3)}{2}}=\prod_{k<l} P_n(y_1,\dots,\hat{y}_k,\dots,\hat{y}_l,\dots,y_n)$$
we hope this relation gives insight into relation to Conjecture 3 in \cite[Page 16]{Ba}. In particular one would like to use lemma \ref{inductive} to obtain inductive proofs of our results, unfortunately, that is not the case. However, it simplifies the proof of Pohst's inequality to prove that a proof of the bound $P_n \leq 2^{\lfloor \frac{n}{2}\rfloor}$ for $n$ odd implies the bound for $n+1$.

Going back to the physical heuristic for this problem, if we have $n$ ordered particles in the positive real line, $y_1,\dots,y_n$, we expect the product $P_n$ to be maximal when putting together particles of opposite sign. Note that if $(1-y_i/y_j)$ is close to $2$ and there is some $i<k<j$ then we necessarily have a two particles $y_{k_1}$ and $y_{k_2}$, $i\leq k_1<k_2\leq j$ with equal sign and hence a factor $(1-y_{k_1}/y_{k_2})$ close to $0$. This motivates the following.

\textbf{Conjecture.}
\textit{Let $\{y_i\}_{i=1}^n \subset \R^\ast$ such that $|y_i|<|y_{i+1}|$ for $i=1,\dots,n-1$. Let $M$ be the maximal number of consecutive disjoint pairs $(y_i,y_{i+1})$ of opposite sign, then} $$P_n(y_1,\dots,y_n)=\prod_{1\leq i<j\leq n}\big(1-\dfrac{y_i}{y_{j}}\big)\leq 2^{M}.$$ 

For example if the sign configuration of the $\{y_i\}_{i=1}^n$ is $(+,+,-,-)$, then $M=1$. Note that if $p$ is the number of positive variables $y_i$ and $m=n-p$, then $M \leq \min(m,p)$.

\textbf{Open problem.} Despite the effort needed to construct a good partition, for a given sign configuration these are not unique: $(-,-,-)$ has $2$ such partitions. The question on how many of these partitions exists for each sign configuration remains open.

\section{Appendix}

Here we include an explicit example for $n=8$, for the coordinates of $\vec{x}$ having signs $(-,+,+,-,+,+,-,-)$. In this example levels $1$, $4$, $5$ and $7$ are unstable. Canonical elements are indicated in red with an index indicating its class $p \in \eta_{\vec{x}}$.

\begin{figure}[H]
    \centering
    \begin{subfigure}[t]{0.5\textwidth}
    \begin{tikzpicture}
    \matrix [matrix of math nodes] (m)
        {
            (1,8)_+^{{\color{red}{7}}} &(2,8)_-^{{\color{red}{7}}} &(3,8)_- &(4,8)_-^{{\color{red}{8}}} &(5,8)_+^{{\color{red}{8}}} &(6,8)_+ & (7,8)_+^{{\color{red}{6}}} &(8,8)_-^{{\color{red}{6}}}  \\
            (1,7)_-^{{\color{red}{7}}} &(2,7)_+^{{\color{red}{7}}} &(3,7)_+ &(4,7)_+^{{\color{red}{5}}} &(5,7)_-^{{\color{red}{5}}} &(6,7)_- & (7,7)_-^{{\color{red}{6}}}\\
            (1,6)_+^{{\color{red}{1}}} &(2,6)_-^{{\color{red}{1}}} &(3,6)_- &(4,6)_-^{{\color{red}{5}}} &(5,6)_+^{{\color{red}{5}}} &(6,6)_+ & &  \\    
            (1,5)_+ &(2,5)_- &(3,5)_-^{{\color{red}{4}}} &(4,5)_- &(5,5)_+ & & & \\      
            (1,4)_+^{{\color{red}{2}}} &(2,4)_-^{{\color{red}{2}}} &(3,4)_- &(4,4)_-^{{\color{red}{3}}} & & & & \\     
            (1,3)_-^{{\color{red}{2}}} &(2,3)_+^{{\color{red}{2}}} &(3,3)_+ & & & & & \\    
            (1,2)_- &(2,2)_+ & & & & & & \\     
            (1,1)_-^{{\color{red}{1}}}\\ 
        };
         \draw[color=red] (m-8-1) ellipse (6.3mm and 2.9mm);
         \draw[color=red] (m-6-1) ellipse (6.3mm and 2.9mm);
         \draw[color=red] (m-6-2) ellipse (6.3mm and 2.9mm);
         \draw[color=red] (m-5-1) ellipse (6.3mm and 2.9mm);
         \draw[color=red] (m-5-2) ellipse (6.3mm and 2.9mm);
         \draw[color=red] (m-5-4) ellipse (6.3mm and 2.9mm);
         \draw[color=red] (m-4-3) ellipse (6.3mm and 2.9mm);
         \draw[color=red] (m-3-1) ellipse (6.3mm and 2.9mm);
         \draw[color=red] (m-3-2) ellipse (6.3mm and 2.9mm);
         \draw[color=red] (m-3-4) ellipse (6.3mm and 2.9mm);
         \draw[color=red] (m-3-5) ellipse (6.3mm and 2.9mm);
         \draw[color=red] (m-2-1) ellipse (6.3mm and 2.9mm);
         \draw[color=red] (m-1-1) ellipse (6.3mm and 2.9mm);
         \draw[color=red] (m-2-2) ellipse (6.3mm and 2.9mm);
         \draw[color=red] (m-1-2) ellipse (6.3mm and 2.9mm);
         \draw[color=red] (m-2-4) ellipse (6.3mm and 2.9mm);
         \draw[color=red] (m-1-4) ellipse (6.3mm and 2.9mm);
         \draw[color=red] (m-2-5) ellipse (6.3mm and 2.9mm);
         \draw[color=red] (m-1-5) ellipse (6.3mm and 2.9mm);
         \draw[color=red] (m-2-7) ellipse (6.3mm and 2.9mm);
         \draw[color=red] (m-1-7) ellipse (6.3mm and 2.9mm);
         \draw[color=red] (m-1-8) ellipse (6.3mm and 2.9mm);
\end{tikzpicture} 
    \end{subfigure}
    \hfill \hfill
    \begin{subfigure}[t]{0.2\textwidth}
        \centering
        \vspace{-5.5cm}
        \hspace{1cm}
    \begin{tabular}{|c|c|c|}
 \hline
 Pair & cs. & op. \\ \hline
  (1,1) & 2 & 3 \\ \hline
  (1,3) & 1 & 1\\ \hline
  (4,4) & 2 & 3 \\ \hline
  (2,4) &3 & 2 \\ \hline
  (3,5) &2 & 3 \\ \hline
  (4,6) &1 & 1 \\ \hline
  (2,6) &5 & 4 \\ \hline
  (7,7)&2 & 3 \\ \hline
  (5,7) & 3 & 2 \\ \hline
  (1,7) & 1 & 1\\ \hline
  (8,8) &5 & 4 \\ \hline
  (4,8) &1 & 1 \\ \hline
  (2,8) & 6 & 2\\ \hline
\end{tabular}
    \end{subfigure}
    \caption{Example of partition $\eta_{\vec{x}}$ and negative pairs with corresponding case (cs.) and operation (op.) used.}
\end{figure}

In the following algorithm the variables $r$ and $t$ are static, that is, their value is preserved through different runs of the algorithm for a fixed sign configuration of $\vec{x}$.

\begin{breakablealgorithm}
\caption{Algorithm to construct $P_k$ from $P_{k-1}$.}
\begin{algorithmic}[1]
\If{$(i,j)_-$ is ideal} \Comment{Case 1}
    \State Let  $(i+l,j)_+$ maximal with order $ \rightarrow$ in $\sing$-configuration.
    \State Use Operation 1 with $(i+l,j)_+$.
\ElsIf{Level $j$ is unstable}
        \If{$(i,j)_-$ is the smallest non ideal on level $j$} \Comment{Case 2}
        \State Use Operation 3. 
        \ElsIf{$(i,j)_-$ is the second smallest non ideal on level $j$} \Comment{Case 3}
        \State Let $(i,j')_+$ a pair in $\sing$-configuration or in a $\hdoub$-configuration (lemma \ref{Case3Possible}).
        \State Record the coordinate $r=j'$.
            \If{$(i,j')_+$ in $\sing$-configuration}
            \State Use Operation $1$ with $(i,j')_+$.
            \ElsIf{$(i,j')_+$ in $\hdoub$-configuration}
            \State Use Operation $2$ with $(i,j')_+$.
            \EndIf
    \Else \Comment{Case 4}
    \State Let $r$ be the value recorded in Case 3.
        \If{$(i,r)_+$ in $\sing$-configuration}
        \State Use Operation $1$ with $(i,r)_+$.
        \ElsIf{$(i,r)_+$ in $\hdoub$-configuration}
        \State Use Operation $2$ with $(i,r)_+$.
        \EndIf
        \EndIf
\ElsIf{Level $j$ is stable}
    \If{$(i,j)_-$ is the smallest non ideal on level $j$} \Comment{Case 5}
    \State Let $(i',i-1)_-$ be a pair in $\nsing$-configuration (lemma \ref{Case5Possible}).
    \State Record the coordinate $t=i-1$.
    \State Use operation 4 with $(i',i-1)_-$ and $(i',j)_+$. 
    \Else \Comment{Case 6}
    \State Let $t$ be the value recorded in Case 5.
        \If{$(i,t)_+$ in $\sing$-configuration}
        \State Use Operation $1$ with $(i,t)_+$.
        \ElsIf{$(i,t)_+$ in $\hdoub$-configuration}
        \State Use Operation $2$ with $(i,t)_+$.
        \EndIf
\EndIf
\EndIf        
\end{algorithmic}
\end{breakablealgorithm}

We have found no example of Case $6$ where operation $1$ is applicable, similarly, the positive pair of line $8$ of the algorithm appear to be unique. Although not proved here, we believe these statements can be proven with the methods of this paper. 
\vspace{3mm}

\noindent\textbf{Acknowledgment} I am indebted to the anonymous referee for the detailed report, suggestions and corrections that significantly improved the presentation of this paper.

\end{document}